\documentclass[twoside]{amsart}
\usepackage[english]{babel}
\usepackage{graphicx}
\usepackage{mathtools}		
\usepackage{mathabx}		
\usepackage{indentfirst}	
\usepackage{amsmath}
\usepackage{amssymb} 
\usepackage{amsthm}
\usepackage{thmtools}
\usepackage{enumitem}		
\usepackage[modulo,right,pagewise]{lineno} 
\usepackage[backref=page,colorlinks=true]{hyperref}	
\usepackage[capitalise]{cleveref} 
\usepackage{tikz}
\usepackage{tikz-cd}
\usepackage[font=footnotesize]{caption}
\usepackage{mathrsfs}

\newcommand{\wtilde}[1]{\widetilde{#1}}

\newcommand{\Hy}{\mathbb{H}}

\newcommand{\bbP}{\mathbb{P}}

\newcommand{\R}{\mathbb{R}}

\newcommand{\Z}{\mathbb{Z}}

\newcommand{\calA}{\mathcal{A}}

\newcommand{\calC}{\mathcal{C}}
\newcommand{\calD}{\mathcal{D}}

\newcommand{\calG}{\mathcal{G}}

\newcommand{\calP}{\mathcal{P}}

\newcommand{\scrC}{\mathscr{C}}
\newcommand{\scrD}{\mathscr{D}}

\newcommand{\scrT}{\mathscr{T}}

\newcommand{\scrV}{\mathscr{V}}

\newcommand{\al}{\alpha}
\newcommand{\gam}{\gamma}

\newcommand{\del}{\delta}

\newcommand{\Del}{\Delta}
\newcommand{\ep}{\epsilon}

\newcommand{\Thet}{\Theta}
\newcommand{\lam}{\lambda}
\newcommand{\Lam}{\Lambda}

\newcommand{\om}{\omega}

\DeclareMathOperator{\Isom}{Isom}

\DeclareMathOperator{\Aut}{Aut}
\DeclareMathOperator{\Inn}{Inn}
\DeclareMathOperator{\Out}{Out}

\DeclareMathOperator{\Mod}{Mod}

\DeclareMathOperator{\Dil}{Dil}

\newcommand{\bs}{\backslash}

\newcommand{\ra}{\rightarrow}

\renewcommand{\epsilon}{\varepsilon}

\newcommand{\groprod}[3]{\left<#1|#2\right>_{#3}}

\newcommand{\corchete}[1]{\left\{{#1}\right\}}
\newcommand{\ov}[1]{\overline{#1}}

\newcommand{\CAT}[1]{\textup{CAT}(#1)}

\newcommand{\Curr}{\calC urr}
\newcommand{\PCurr}{\bbP\calC urr}

\usepackage{color}
\usepackage[colorlinks=true]{hyperref}
\usepackage{fancyhdr}
\usepackage[left=3.6cm,top=3.5cm,right=3.6cm,bottom=2.5cm]{geometry}

\renewcommand*{\backref}[1]{}
\renewcommand*{\backrefalt}[4]{\quad \tiny 
  \ifcase #1 (\textbf{NOT CITED.})%
  \or    (Cited on page~#2.)%
  \else   (Cited on pages~#2.)%
  \fi}

\makeatletter

\def\MRbibitem{\@ifnextchar[\my@lbibitem\my@bibitem}

\def\mybiblabel#1#2{\@biblabel{{\hyperref{http://www.ams.org/mathscinet-getitem?mr=#1}{}{}{#2}}}}

\def\myhyperanchor#1{\Hy@raisedlink{\hyper@anchorstart{cite.#1}\hyper@anchorend}}

\def\my@lbibitem[#1]#2#3#4\par{%
  \item[\mybiblabel{#2}{#1}\myhyperanchor{#3}\hfill]#4%
  \@ifundefined{ifbackrefparscan}{}{\BR@backref{#3}}%
  \if@filesw{\let\protect\noexpand\immediate
    \write\@auxout{\string\bibcite{#3}{#1}}}\fi\ignorespaces%
}

\def\my@bibitem#1#2#3\par{%
  \refstepcounter\@listctr
  \item[\mybiblabel{#1}{\the\value\@listctr}\myhyperanchor{#2}\hfill]#3%
  \@ifundefined{ifbackrefparscan}{}{\BR@backref{#2}}%
  \if@filesw\immediate\write\@auxout
    {\string\bibcite{#2}{\the\value\@listctr}}\fi\ignorespaces%
}

\makeatother

\newtheorem{thm}{Theorem}[section]
\newtheorem{prop}[thm]{Proposition}
\newtheorem{lemma}[thm]{Lemma}
\newtheorem{coro}[thm]{Corollary}
\theoremstyle{remark}\newtheorem{rmk}[thm]{Remark}
\theoremstyle{definition}

\newtheorem{defi}[thm]{Definition}

\begin{document}

\title{\textbf{The space of metric structures on hyperbolic groups}}

\author{\small{Eduardo Oreg\'on-Reyes}}

\markboth{E Oreg\'on Reyes}{Nombreartículo}

\date{}
\maketitle

\begin{abstract} We study the metric and topological properties of the space $\scrD(G)$ of left-invariant hyperbolic pseudometrics on the non-elementary hyperbolic group $G$ that are quasi-isometric to a word metric, up to rough similarity. This space naturally contains the Teichmüller space in case $G$ is a surface group and the Culler-Vogtmann outer space when $G$ is a free group. Endowed with a natural metric reminiscent of the (symmetrized) Thurston's metric on Teichmüller space, we prove that $\scrD(G)$ is an unbounded contractible metric space and that $\Out(G)$ acts metrically properly by isometries on it. If we restrict ourselves to the subspace $\scrD_\del(G)$ of the points represented by $\del$-hyperbolic pseudometrics with critical exponent 1, we prove that it is either empty or proper. We also prove continuity of the Bowen-Margulis map from $\scrD_\del(G)$ into the space $\bbP\Curr(G)$ of projective geodesic currents on $G$, extending similar results for surface and free groups, and the continuity of the (normalized) mean distortion as a function on $\scrD(G)\times \scrD(G)$.
\end{abstract}

\paragraph{\hspace{12.5mm}\textbf{Mathematics Subject Classification (2010).}} 20F67, 20F65, 51F30.

\section{Introduction}
Let $G$ be a non-elementary hyperbolic group. If $G$ is either a surface group or a non-abelian free group, there are contractible spaces parametrizing nice geometric structures on $G$, namely, the Teichmüller space $\scrT(G)$ and the Culler-Vogtmann outer space $\scrC\scrV(G)$ \cite{CullerVogtmann1986Outerspace}, respectively. These spaces have been studied extensively, and nowadays they (and their generalizations) are standard tools in the understanding of mapping class groups, outer automorphism groups of free groups, and related notions in group theory and low-dimensional topology \cite{Bestvina2002ICM,DahmaniHorbez2018RandomwalksMCG+Out,FarbMargalit2012primer,GuirardelHorbez2021MeasEquivRigOut(Fn),Otal1998HypHaken,Wienhard2019ICM}.

The situation differs dramatically in higher dimensions, since by Mostow rigidity, if $G$ is the fundamental group of a closed aspherical manifold $M$ of dimension $n\geq 3$, then there exists at most one hyperbolic structure on $M$. This gets worse in dimensions $n\geq 5$, as there are examples of such $M$ admitting no negatively curved Riemannian metrics \cite[Thm.~5c.1 and Rmk.~on p.~386]{DavisJanuszkiewicz1991hyppolyhedra}.

Besides examples coming from geometric topology or representation theory, finding interesting deformation spaces for arbitrary hyperbolic groups is in general complicated, and most constructions rely on the existence of non-trivial isometric actions on $\R$-trees (see e.g.~\cite{
Clay2005ContractdefspaceGtrees,GuirardelLevitt2007Outerspacefreeproduct,Paulin1988TopGromovRtrees}). However, since Gromov hyperbolicity is a metric property, it is natural to consider structures associated to interesting isometric actions of $G$ on hyperbolic metric spaces, under the appropriate equivalence relation. This coarse-geometric perspective was adopted in \cite{Furman2002coarsenegcurv} by Furman, where he considered the space of all hyperbolic, left-invariant metrics on $G$ that are quasi-isometric to a word metric with respect to a symmetric and finite generating set of $G$. 

Since we also want to include hyperbolic groups with torsion, it is more convenient to consider the space $\calD(G)$ of all hyperbolic, left-invariant \emph{pseudometrics} that are quasi-isometric to a word metric. The relevant object is then the quotient of $\calD(G)$ under rough similarity, where two pseudometrics $d,d'$ on a set $X$ are \emph{roughly similar} if there exist $k,A>0$ such that $|d-k\cdot d'|\leq A$.
\begin{defi}
Let $G$ be a non-elementary hyperbolic group. The \emph{space of metric structures} on $G$ is $\scrD(G)$, the set of equivalence classes of pseudometrics in $\calD(G)$, where two pseudometrics are in the same class if they are roughly similar.

Points in $\scrD(G)$ will be called \emph{metric structures}, and we denote by $[d]$ the metric structure induced by $d\in \calD(G)$.
\end{defi}
Metrics structures on $G$ can be defined from different sources, such as word metrics on $G$, Green metrics associated to symmetric finitely supported and generating random walks on $G$ \cite{BlachereHaissinskyMathieu2011HarmvsQconf}, geometric actions of $G$ on $\CAT{0}$ cube complexes, non-positively curved Riemannian metrics on the closed manifold $M$ in case $G=\pi_1(M)$, Anosov representations of $G$ in higher rank simple Lie groups \cite{DeyKapovich2019PSAnosov}, and more generally from any geometric action of $G$ on a proper geodesic metric space. 

Including pseudometrics to the definition of $\calD(G)$ allows metric structures induced by geometric actions that are not necessarily free. In any case, every metric structure is represented by a genuine metric, since if $d'\in \calD(G)$ is a pseudometric, then for any $\ep>0$ the metric $d_\ep(x,x)=0$ and $d_\ep(x,y):=d'(x,y)+\ep$
for $x\neq y \in G$ is in the same rough similarity class of $d$. Also, if $G$ is torsion free, then every pseudometric in $\calD(G)$ is a metric, see Remark \ref{torsionfree}. 

By the work of Coornaert \cite{Coornaert1993PSmeasures}, Furman \cite{Furman2002coarsenegcurv}, and Bader and Furman \cite{BaderFurman2017Ergodichyp}, a metric structure $\rho=[d]$ can be recovered from various objects related to $d$. These include:
\begin{itemize}
    \item The homothety class of the \emph{marked length spectrum} of $d$, which is the function $\ell_d:G \ra \R$ given by
    $$\ell_d(x)=\lim_{n\to \infty}{\frac{d(x^n,1)}{n}}, \hspace{2mm} \text{ for }x\in G.$$
    \item The measure class of any \emph{quasiconformal measure} for $d$ on the Gromov boundary $\partial G$, introduced by Coornaert \cite{Coornaert1993PSmeasures}.
    \item The \emph{projective Bowen-Margulis current} $BM(\rho)$ on the double boundary $\partial^2 G$.
\end{itemize}
From these characterizations we get natural injective maps from the Teichmüller space $\scrT(G)$ into $\scrD(G)$, provided that $G=\pi_1(S)$ is a surface group. Moreover, by the solution of the Marked Length Spectrum Rigidity conjecture for negatively curved metrics \cite{Otal1990MLSRigidity}, $\scrD(G)$ contains a copy of the space of (homothety classes of) all negatively curved metrics on $S$ up to isotopy. Similarly, when $G$ is a non-abelian free group, there is a canonical injective map from the outer space $\scrC\scrV(G)$ into $\scrD(G)$ \cite{FrancavigliaMartino2011Metricouter}.

The space $\scrD(G)$ is equipped with a natural metric $\Del$, which measures how far are two quasi-isometric pseudometrics from being roughly similar.
\begin{defi}
Given $\rho_1=[d_1],\rho_2=[d_2]\in \scrD(G)$, we define
$$\Lambda(\rho_1,\rho_2):=\inf\corchete{\lam_1\lam_2 \colon \exists A\geq 0 \text{ s.t. }\frac{1}{\lam_1}d_2-A\leq d_1\leq \lam_2d_2+A},$$
and $$\Del(\rho_1,\rho_2):=\log\Lam(\rho_1,\rho_2).$$
\end{defi}
It is clear that $\Lam$ and $\Del$ are independent of the representatives $d_1$ and $d_2$, and that $\Del$ is nonnegative, symmetric and satisfies the triangular inequality. If $\Del(\rho_1,\rho_2)=0$ then $d_1$ and $d_2$ have homothetic marked length spectra, and hence $\rho_1=\rho_2$ (see also \cite[Thm.~1.1]{Krat2001Pairsmetrics} or \cite[Thm.~1.2]{CantrellTanaka2021Manhattan}), implying that  $(\scrD(G),\Del)$ is a metric space. 

The metric $\Del$ is not at all new, and can be recovered from the marked length spectra of the metric structures (see Proposition \ref{Lambda=DilDil} below). Indeed, when $G=\pi_1(S)$ is a surface group, the restriction of $\Del$ to $\scrT(S)$ coincides with the (symmetrized) Thurston's metric \cite{Thurston1998Minstretchmaps}, and hence the injection $\scrT(S) \ra \scrD(G)$ is continuous (see Remark \ref{Thurstonmetric}). The work of Francaviglia and Martino \cite{FrancavigliaMartino2011Metricouter} implies the same result in the case of free groups, so that the injection $\scrC\scrV(F_k) \ra \scrD(F_k)$ is also continuous.

There are some results regarding the intersection of interesting subsets of $\scrD(G)$. For example, Gouëzel, Math\'eus and Maucourant proved that if $G$ admits a word metric $d_S$ and a finitely supported symmetric Green metric $d_\mu$ such that $[d_S]=[d_\mu]$, then $G$ is virtually free \cite[Thm.~1.3]{GouezelMatheusMaucourant2018Entropydrifthyp}. In a similar spirit, Fricke and Furman showed that if $G=\pi_1(S)$ is a surface group, then the Teichmüller space $\scrT(S)$ is precisely the intersection of the set of metric structures coming from negatively curved metrics on $S$ and the space of metric structures induced by quasi-Fuchsian representations of $G$ \cite[Thm.~A]{FrickerFurman2021QFuchsvsNegcurv}. A famous open problem states that if $G$ is a cocompact lattice in the real hyperbolic space $\Hy^n$ and $d_\mu$ is a Green metric associated to finitely supported symmetric random walk on $G$, then $[d_{\Hy^n}]\neq [d_\mu]$. Recently, there has been some progress in the 2-dimensional case, see \cite{KosenkoTiozzi2021FundineqFuchsian}.

The goal of this paper is to start the study of $(\scrD(G),\Del)$ from a global and geometric perspective. Now we state our results.

\subsection{Contractibility and properness}
Our first result is an analogue of a well-known property for Teichmüller space and outer space.
\begin{thm}\label{unboundedcontractible}
The metric space $(\scrD(G),\Del)$ is unbounded and contractible.
\end{thm}
In similar settings, contractibility is in general a non trivial task, but in our case it will follow almost immediately from the fact that $\calD(G)$ is closed under addition of pseudometrics. Unboundedness is interesting, since it holds even when $\Out(G)$ is finite. We exhibit an unbounded sequence in $\scrD(G)$ by finding metrics that (almost) kill arbitrarily large powers of some infinite order element of $G$, with some similarity to group theoretical Dehn filling.

To state our properness results, in principle we must restrict to smaller subspaces of $\scrD(G)$. 
\begin{defi}
\begin{itemize}
    \item $\scrD_\del(G)$ is the space of all metric structures $\rho$ such that there is some $\del$-hyperbolic pseudometric $d\in \rho$ with critical exponent 1.
    \item $\scrD_{\del,\al}(G)$ is the space of all metric structures $\rho$ such that there is some $\del$-hyperbolic and $\al$-rough geodesic pseudometric $d\in \rho$ with critical exponent 1.
\end{itemize}
\end{defi}
See Section \ref{secprelim} for the relevant definitions. When $G$ is a non-abelian free group, $\scrD_0(G)$ is exactly the outer space $\scrC \scrV(G)$. Similarly, since $\Hy^2$ is $\log{2}$-hyperbolic \cite[Cor.~5.4]{NicaSpakula2016stronghyp} and the critical exponent of any cocompact lattice in $\Hy^2$ is 1, when $G$ is a surface group we have the injection $\scrT(G)\hookrightarrow \scrD_{\log{2}}(G)$. Properness of $\scrT(G)$ and $\scrC\scrV(G)$ are then extended by the next result.

\begin{thm}\label{proper}
For any $\del, \al \geq 0$ we have:
\begin{itemize}
    \item[i)] $\scrD_{\del}(G)$ is either empty or a proper subspace of $(\scrD(G),\Del)$, and;
    \item[ii)] $\scrD_{\del,\al}(G)$ is either empty or proper.
\end{itemize}

\end{thm}
 Recall that a metric space is proper if its closed balls are compact. In particular, $\scrD_\del(G)$ and $\scrD_{\del,\al}(G)$ are always complete subspaces of $\scrD(G)$.
 
The proof Theorem \ref{proper} splits into three parts: the first step is to prove that $\scrD_{\del,\al}(G)$ is complete for any $\del\geq 0$, which requires to construct a limit for a given Cauchy sequence of metric structures. A key point here is to find good representatives for such metric structures, so that we can extract a pointwise convergent subsequence. This is achieved by applying a Bochi-type inequality for subsets of isometries in Gromov hyperbolic spaces due to Breuillard and Fujiwara \cite[Thm.~1.4]{BreuillardFujiwara2021JSRnonpos}. From this inequality we deduce Lemma \ref{lem.unifquasiisom}, our main tool to find good representatives for metric structures. Once a limit pseudometric has been produced, we check that it actually defines a metric structure in $\scrD_{\del,\al}(G)$, which is the desired limit of our Cauchy sequence. 

Our second step is to verify that $\scrD_{\del,\al}(G)$ is proper, for which it is enough to show that any bounded subset $B \subset \scrD_{\del,\al}(G)$ can be covered by finitely many sets of arbitrarily small diameter. To achieve this, we apply \cite[Thm.~1.4]{BreuillardFujiwara2021JSRnonpos} to every metric structure in $B$, so that they can be uniformly approximated by metric structures induced by word metrics (here the $\al$-rough geodesic assumption is crucial). Balls of uniformly small radius around these ``word" metric structures will cover $B$, and indeed, we can cover $B$ with only finitely many of such balls.

Finally, to check that $\scrD_\del(G)$ is proper, we prove that any bounded subset $B\subset \scrD_\del(G)$ is contained in $\scrD_{\del,\al}(G)$ for some $\al$. Again, we want to apply \cite[Thm.~1.4]{BreuillardFujiwara2021JSRnonpos} to $\del$-hyperbolic pseudometrics representing metric structures in $B$, but first we need a $G$-equivariant isometric embedding of $(G,d)$ into a \emph{geodesic} $\del$-hyperbolic space. For this we use the embedding obtained from the injective hull functor \cite{Lang2013injectivehull}, and an application of Morse lemma in the form of Lemma \ref{lem.roughgeodfromhyp} then allows us to find $\del$-hyperbolic pseudometric representatives that are uniformly roughly geodesic.

\subsection{The action of $\Out(G)$ on $\scrD(G)$}
One of the main properties of Teichmüller space and outer space, is that they admit proper isometric actions of $\Mod(S)\cong \Out(\pi_1(S))$ and $\Out(F_k)$ respectively. In our more general setting, there is a natural isometric action of $\Out(G)$ on $(\scrD(G),\Del)$ via pullback: if $\psi\in \Aut(G)$ and $d\in \calD(G)$, then $\psi([d])=[\psi(d)]$, where $\psi(d)(x,y)=d(\psi^{-1}x,\psi^{-1}y)$ for all $x,y\in G$. This action is clearly isometric for the metric $\Del$, and if $\psi\in \Inn(G)$ then $\psi$ acts trivially on $\scrD(G)$, so that the action descends to an isometric action of $\Out(G)$. For this action we prove the following.

\begin{thm}\label{Outmetricproper}
The action of $\Out(G)$ on $\scrD(G)$ is isometric and metrically proper. That is, for any $\rho\in \scrD(G)$ and $R\geq 0$, there exist at most finitely many $\psi\in \Out(G)$ such that $\Del(\rho, \psi(\rho))\leq R.$
\end{thm}

The proof of Theorem \ref{Outmetricproper} also uses Breuillard-Fujiwara's inequality, combined with an argument of Paulin in his proof that hyperbolic groups with infinite outer automorphism group act non-trivially on $\R$-trees \cite{Paulin1991OuthypRtrees}.

The situation is also interesting when we restrict to the $\Out(G)$-invariant subspaces $\scrD_\del(G)$ and $\scrD_{\del,\al}(G)$ for $\del,\al\geq 0$, since by Theorem \ref{proper}, metric properness of the action implies proper discontinuity. When $G$ is torsion-free, we can apply a finiteness theorem for hyperbolic groups recently obtained by Besson-Courtois-Gallot-Sambusetti \cite{BessonCourtoisGallotSambusetti2021Finitenesshyp}, and actually prove cocompactness for the action on $\scrD_{\del,\al}(G)$.
\begin{thm}\label{cocompactness}
Assume $G$ is torsion-free, and let $\al, \del\geq 0$ be such that $\scrD_{\del,\al}(G)$ is non-empty. Then the action of $\Out(G)$ on $\scrD_{\del,\al}(G)$ is cocompact. In particular, if $\Out(G)$ is finite then $\scrD_{\del,\al}(G)$ is compact.
\end{thm}

\subsection{Geodesic currents and the Bowen-Margulis map} As $G$ is non-elementary, its boundary at infinity $\partial G$ is naturally a compact metrizable space equipped with a topological action of $G$. A \emph{geodesic current} of $G$ is a $G$-invariant Radon measure on $\partial^2G$, the set of pairs of distinct elements of $\partial G$ considered with the diagonal action. Let $\Curr(G)$ denote the space of geodesic currents of $G$ (some authors also require currents to be flip-invariant, but we will not assume this). 

The space $\Curr(G)$ is considered with the weak-$*$ topology, so that a sequence $(m_n)_n$ of geodesic currents converges to the geodesic current $m$ if and only if $\int{fdm_n} \to \int{fdm}$ for every compactly supported continuous function $f\in C_c(\partial ^2 G)$. We will also consider the quotient space $\bbP\Curr(G):=(\Curr(G)\bs \corchete{0})/\R^+$ of \emph{projective geodesic currents}, endowed with the quotient topology, which makes it into a compact metrizable space \cite[Prop.~6]{Bonahon1991currentsnegcurv}.

Geodesic currents were introduced by Bonahon \cite{Bonahon1988GeomTeichcurrents}, and were used to give a new interpretation of Thurston's compactification for Teichmüller spaces. In case $G$ is the fundamental group of a closed negatively curved manifold $M$, $\partial ^2 G$ is naturally homeomorphic to set of unparametrized oriented geodesics in the universal cover $\wtilde{M}$, and geodesic currents can be described as Borel measures on the unit tangent bundle of $M$ that are invariant for the geodesic flow \cite{Kaimanoch1990Invmeasgeodesicflow}.  

Given $\rho \in \scrD(G)$, we can construct a projective geodesic current $BM(\rho)\in \bbP\Curr(G)$ as follows. If $\rho=[d]$ and $\nu$ is a  quasiconformal measure for $d$, there exists a geodesic current $m$ in the same measure class of $\nu \otimes \nu $, and any other geodesic current in this measure class is a positive multiple of $m$. This induces a projective geodesic current, which turns out to be independent of the choices of $d$ and $\nu$ (see \cite[Prop.~3.1]{Furman2002coarsenegcurv}). 
This is the \emph{Bowen-Margulis current} of the metric structure $\rho$.

In \cite[Thm.~4.1]{Furman2002coarsenegcurv}, Furman proved that the Bowen-Margulis map $BM:\scrD(G)\ra \bbP\Curr(G)$ is injective for any non-elementary hyperbolic group $G$. If $G=\pi_1(S)$ is a surface group and $\rho \in \scrT(S)$, the Bowen-Margulis current $BM(\rho)$ has a canonical normalization in $\Curr (G)$, which coincides with the Liouville current of $\rho$. The work of Bonahon \cite{Bonahon1988GeomTeichcurrents} then implies that the Bowen-Margulis map $BM:\scrT(S) \ra \PCurr (G)$ is continuous. Similar continuity results for the Bowen-Margulis current still hold if we consider variable negative curved metrics on surfaces \cite{LatokKnieperPollicottWeiss1989DiffentropyAnosov} (indeed, sufficiently regular perturbations of the metric will give higher regularity of the Bowen-Margulis current). In the case of a free group $F_k$, the continuity of $BM: \scrC\scrV(F_k) \ra \bbP\Curr (F_k)$ is due to Kapovich and Nagnibeda \cite[Thm.~A]{KapovicNagnibeda2007PSembedding+minvol}. Our next theorem generalizes all these previous continuity results.
\begin{thm}\label{BMctns}
For any $\del\geq 0$ such that $\scrD_\del(G)$ is non-empty, the Bowen-Margulis map $BM:\scrD_\del(G) \ra \PCurr(G)$ is continuous. 
\end{thm}
The tools involved in the proof of Theorem \ref{BMctns} are similar to the ones used in the proof of Theorem \ref{proper}. If $(\rho_n)_n \subset \scrD_\del(G)$ is a sequence of metric structures converging to $\rho_\infty$, $\del$-hyperbolicity together with Breuillard-Fujiwara's inequality (in the form of Lemma \ref{lem.unifquasiisom}) allow us to find convenient pseudometric representatives, so that for each $n$ we can find a well-controlled quasiconformal measure. We prove that any limit of a convergent subsequence of such measures is quasiconformal for $\rho_\infty$, which allows us to find well-controlled Radon measures in $\Curr(G)$ representing $BM(\rho_n)$ and convergent (after taking a subsequence) to a measure representing $BM(\rho_\infty)$.

\subsection{Mean distortion} Given two pseudometrics $d, d'$ in $\calD(G)$, the \emph{mean distortion} of $d'$ over $d$ is the quantity
\begin{equation}\label{meandist}
    \tau(d'/d)=\lim_{r\to \infty}{\frac{1}{|\corchete{x\in G \colon d(x,1)\leq r}|}{\sum_{d(x,1)\leq r}{\frac{d'(x,1)}{r}}}}.
\end{equation}
This limit is well-defined \cite[Thm.~1.2]{CantrellTanaka2021Manhattan}, finite, and positive. The mean distortion has been studied in the case that $d$ and $d'$ are word metrics \cite{CalegariFujiwara2010Combable}, and appears in the study of automorphisms of hyperbolic groups as the generic stretching factor \cite{KaimanovichKapovichSchupp2007Genstretchingfactor}. 

The mean distortion is also related to the Hausdorff dimension of harmonic measures. If $d\in \calD(G)$, $\mu$ is a finitely supported probability measure on $G$ with symmetric support that generates $G$, and $d_{\mu}$ is its corresponding Green metric, then
$$1/\tau(d/d_{\mu})=h_{\mu}/l_\mu,$$
where $h_{\mu}$ is the entropy of $\mu$ and $l_\mu$ is the drift of $\mu$ with respect to $d$. The identity follows since both sides coincide with the Hausdorff dimension of the harmonic measure for $\mu$ with respect to a visual quasi-metric for $d$ on $\partial G$, see e.g. \cite[Cor.~1.4]{BlachereHaissinskyMathieu2011HarmvsQconf} and \cite[Thm.~3.8]{CantrellTanaka2021Manhattan}.

Given metric structures $\rho,\rho'\in \scrD(G)$, we define the \emph{mean distortion} of $\rho'$ over $\rho$ by
$$\tau(\rho'/\rho)=\tau(d'/d),$$
where $d\in \rho$ and $d'\in \rho$ have critical exponent 1.
It follows from \cite[Thm.~1.2]{CantrellTanaka2021Manhattan} that $\tau(\rho'/\rho)\geq 1$, with equality if and only if $\rho=\rho'$. It is clear from \eqref{meandist} that $\tau(\rho'/\rho)$ is continuous in the variable $\rho'$, and the next theorem states that it is actually continuous in both variables.
\begin{thm}\label{meandistctns}
The mean distortion $\tau:\scrD(G)\times \scrD(G) \ra \R$ is continuous. 
\end{thm}
The proof of this result is short, and will follow from the work of Cantrell and Tanaka \cite{CantrellTanaka2021Manhattan}, where they describe the mean distortion of $d'$ over $d$ in terms of the derivative of the Manhattan curve of the pair $(d,d')$.


\paragraph{\textbf{Organization of the paper.}} Section \ref{secprelim} presents the main properties of Gromov hyperbolic spaces and groups that we use throughout the paper. In section \ref{secMLS} we prove Proposition \ref{Lambda=DilDil}, which characterizes the distance $\Del$ in terms of the marked length spectra of the corresponding metric structures. The main tool is Proposition \ref{thmDGLM}, a refinement of a result by Delzant-Guichard-Labourie-Mozes. The proof of Theorem \ref{unboundedcontractible}, split into Propositions \ref{contractible} and \ref{unbounded}, is given in Section \ref{seccontractibleunbounded}. In Section \ref{secproperness} we prove Theorem \ref{proper}, divided into Propositions  \ref{complete}, \ref{propersecond} and \ref{boundimpliesroughgeod}. For this we require Lemma \ref{worddense}, which allows us to approximate pseudometrics in $\calD(G)$ by word metrics, and Lemma \ref{lem.unifquasiisom}, which is used to find good pseudometric representatives for metric structures. This last lemma is consequence Breuillard-Fujiwara's inequality, stated as Theorem \ref{thmbrefuj}. The action of $\Out(G)$ on $\scrD(G)$ is studied in Section \ref{secOutG}, where we prove Theorems \ref{Outmetricproper} and \ref{cocompactness}. This last theorem uses Theorem \ref{thmBCGS}, a particular case of the finiteness theorem of Besson-Courtois-Gallot-Sambusetti. We end with Section \ref{secBM} where we prove Theorems \ref{BMctns} and \ref{meandistctns}. 

\paragraph{\textbf{Convention.}}
Throughout the paper, we will deal with properties and arguments regarding pseudometric spaces. Besides most of our cited references are statements about metric spaces, they also hold for pseudometric spaces, either by considering the metric identifications, or by approximating pseudometrics by appropriate metrics. For this reason, most of the cited references will be stated in terms of pseudometric spaces without further explanation.


\section{Preliminaries}\label{secprelim}
We start with the basic notions about Gromov hyperbolic spaces and groups, standard references being \cite{BridsonHaefliger1999,GhydelaHarpe1990book}.
Our goal here is twofold: we want to standardize notation, as well as collect known results that will be useful to us given their explicit dependence on parameters such as the hyperbolicity constant, the rough geodesic constant, etc.  The reader familiar with these topics might want to skip to the next section.  

\subsection{Gromov hyperbolicity and quasi-isometries}
Let $(X,d)$ be a pseudometric space and $w\in X$ a base point. The \emph{Gromov product} on $X$ is given by
\begin{equation}
2\groprod{x}{y}{w,d}=d(x,w)+d(y,w)-d(x,y). \notag
\end{equation}
If there is no ambiguity in the pseudometric $d$, we just write $\groprod{x}{y}{w}$ instead of $\groprod{x}{y}{w,d}$.
\begin{defi}\label{defhyp}
The pseudometric space $(X,d)$ is $\delta$-\textit{hyperbolic} ($\del \geq 0$) if for all $x,y,z,w\in X$ the following inequality is satisfied:
\begin{equation}\label{gromovin}
\groprod{x}{z}{w}\geq \min\left(\groprod{x}{y}{w},\groprod{y}{z}{w}\right)-\delta. 
\end{equation}
A pseudometric space is \emph{hyperbolic} if it is $\del$-hyperbolic for some $\del\geq 0$.
\end{defi}
Recall that a pseudometric space is \emph{geodesic} if every two points can be joined by an arc isometric to the interval of length equal to the distance between the two points. In general, we will not require pseudometric spaces to be geodesic. However, most of our pseudometrics will satisfy a weaker (but still useful enough) property:
\begin{defi}\label{defroughgeod}
The pseudometric space $(X,d)$ is $\al$-\emph{rough geodesic} ($\al\geq 0$) if for any $x,y\in X$ there is a sequence of points $x=x_0,\dots,x_n=y$ in $X$ such that for all $0\leq i,j\leq n$:
\begin{equation*}
    |j-i|-\al \leq d(x_j,x_i)\leq |j-i|+\al.
\end{equation*}
\end{defi}

Given constants $\lambda_1,\lambda_2>0$ and $\epsilon \geq 0$, a map $F:X\ra Y$ is a $(\lambda_1,\lambda_2,\epsilon)$-\emph{quasi-isometric embedding} if for any $x,y\in X$ we have
\begin{equation*}
\frac{1}{\lambda_1}d_X(x,y)-\epsilon \leq  d_Y(Fx,Fy)\leq \lambda_2 d_X(x,y)+\epsilon.
\end{equation*}
If in addition, the image of $F$ is $C$-dense in $Y$ for some $C\geq 0$, we say that $F$ is a $(\lam_1,\lam_2,\ep)$-\emph{quasi-isometry}. 

In general, quasi-isometries do not preserve hyperbolicity, unless the spaces are rough geodesic. Under that assumption, quasi-isometries are also close to preserve Gromov products. The following proposition appears in \cite[Prop.~15 (i), Ch.~5]{GhydelaHarpe1990book} in the case of geodesic metric spaces, and can be easily adapted to the rough geodesic setting:
\begin{prop}\label{strongqi}
For all $\al,\del,\ep\geq 0$ and $\lam_1,\lam_2>0$ there exists some $A=A(\al,\del,\lam_1,\lam_2,\ep)\geq 0$ such that the following holds. Let $(X,d_X), (Y,d_Y)$ be $\del$-hyperbolic and $\al$-rough geodesic pseudometric spaces, and $F:X\ra Y$ a $(\lam_1,\lam_2,\ep)$-quasi-isometric embedding. Then for all $x,y,w\in X$:
\begin{equation*}
    \frac{1}{\lam_1}\groprod{x}{y}{w,d_X}-A\leq \groprod{Fx}{Fy}{Fw,d_Y}\leq \lam_2\groprod{x}{y}{w,d_X}+A.
\end{equation*}
\end{prop}

Sometimes, we want to conclude properties for hyperbolic spaces as if they were geodesic or rough geodesic. For our purposes, we consider the \emph{injective envelope} functor \cite{Lang2013injectivehull}, which given a metric space $(X,d)$, assigns another metric space $(\hat{X},\hat{d})$ and an isometric embedding $(X,d)\xhookrightarrow{i} (\hat{X},\hat{d})$. The properties that we need are that $(\hat{X},\hat{d})$ is geodesic, $\del$-hyperbolic if $(X,d)$ is $\del$-hyperbolic \cite[Prop.~1.3]{Lang2013injectivehull}, and that any isometry of $X$ extends uniquely under $i$ to an isometry of $\hat{X}$, which follows from functoriality \cite[Prop.~3.7]{Lang2013injectivehull}. If $(X,d)$ is just a pseudometric space, we can consider the injective hull of its metric identification. We summarize these properties into the following proposition.

\begin{prop}\label{prop.injhull}
If $(X,d)$ is a $\del$-hyperbolic pseudometric space, then its injective hull $(\hat{X},\hat{d})$ is $\del$-hyperbolic and geodesic, and there is an isometric map $i:(X,d) \ra (\hat{X},\hat{d})$. In addition, every isometry of $(X,d)$ extends uniquely to an isometry of $(\hat{X},\hat{d})$, and the map $i$ is equivariant with respect to the group of isometries of $(X,d)$. 
\end{prop}

We end this subsection with an application of the proposition above, which will be needed in Section \ref{secproperness}. A $(\lam_1,\lam_2,\ep)$-\emph{quasigeodesic} in a pseudometric space $(X,d)$ is a $(\lam_1,\lam_2,\ep)$-quasi-isometric embedding $\gam:\corchete{0,\dots,n}\ra X$, where $n$ is a nonnegative integer and $\corchete{0,\dots,n}$ is endowed with its usual distance induced by $\Z$. The pseudometric space $(X,d)$ is $(\lam_1,\lam_2,\ep)$-\emph{quasigeodesic} if for any two points $x,y\in X$ there is a $(\lam_1,\lam_2,\ep)$-quasigeodesic $\gam:\corchete{0,\dots,n}\ra X$ with $\gam(0)=x$ and $\gam(n)=y$. The next lemma ensures that hyperbolic and uniformly quasigeodesic pseudometric spaces are rough geodesic with controlled rough geodesic constant.

\begin{lemma}\label{lem.roughgeodfromhyp}
For any $\lam_1,\lam_2>0$ and $\ep,\del\geq 0$ there exists some $\al=\al(\lam_1,\lam_2,\ep,\del)\geq 0$ such that if $(X,d)$ is a $(\lam_1,\lam_2,\ep)$-quasigeodesic and  $\del$-hyperbolic metric space, then $(X,d)$ is $\al$-rough geodesic.
\end{lemma}
\begin{proof}
Without loss of generality, we can assume that $(X,d)$ is a metric space. By Proposition \ref{prop.injhull}, we consider $(X,d)$ as a subset of its injective envelope $(\hat{X},\hat{d})$, which is $\del$-hyperbolic and geodesic, and let $(X,d) \xhookrightarrow{i} (\hat{X},\hat{d})$ be the inclusion map. Let $\gam:\corchete{0,\dots,n}\ra X$ be a $(\lam_1,\lam_2,\ep)$-quasigeodesic, and apply the Morse lemma to obtain a constant $D=D(\lam_1,\lam_2,\ep,\del)\geq 0$ and a geodesic $\gam':[0,d(\gam(0),\gam(n))] \ra \hat{X}$ joining $\gam(0)$ and $\gam(n)$ such that the Hausdorff distance between $\gam(\corchete{0,\dots,n})$ and $\gam'([0,d(\gam(0),\gam(n))])$ is at most $D$. If $j\in [0,d(\gam(0),\gam(n))]$ is an integer, then there is some $i_j\in \corchete{0,\dots,n}$ such that $\hat{d}(\gam'(j),\gam(i_j))\leq D$, for which we can assume $i_0=0$. In this way, if we define $m=\lfloor{\hat{d}(\gam(0),\gam(n))}\rfloor$, $x_j=\gam(i_j)$ for $0\leq j\leq m$ and $x_{m+1}=\gam(n)$, the sequence $x_0,x_1,\dots,x_{m},x_{m+1}$ is a $(2D+1)$-rough geodesic in $(X,d)$ joining $\gam(0)$ and $\gam(n)$. As every pair of points in $X$ can be joined by a $(\lam_1,\lam_2,\ep)$-quasigeodesic, the conclusion follows by taking $\al=2D(\lam_1,\lam_2,\ep,\del)+1$. 
\end{proof}

\subsection{Boundary at infinity} If $(X,d)$ is $\del$-hyperbolic and $w\in X$, there is a well-defined \emph{Gromov boundary} $\partial X$ consisting of the equivalence classes of sequences $(x_n)_n$ in $X$ such that $\groprod{x_m}{x_n}{w} \to \infty$ as $m,n\to \infty$, where two sequences $(x_n)_n$ and $(y_n)_n$ represent the same point at infinity if $\groprod{x_n}{y_n}{w} \to \infty$. 

The Gromov product can be extended to $X\cup \partial X$ via
$$\groprod{p}{q}{w}:= \sup\corchete{\liminf_{n\to \infty}{\groprod{p_n}{q_n}{w}} \colon p_n \to p, q_n \to q  }.$$
In this way, by \eqref{gromovin} we get that for all $p,q,r\in X\cup \partial X$ and $w\in X$:
\begin{equation}\label{gromovinboundary}
\groprod{p}{r}{w}\geq \min\left(\groprod{p}{q}{w},\groprod{q}{r}{w}\right)-3\delta. 
\end{equation}
We can also define a \emph{Busemann function} on $X\times X\times \partial X$ according to 
$$\beta_d(x,y;p):=\sup\corchete{\limsup_{n \to \infty}
{(d(x,p_n)-d(y,p_n)) \colon p_n\to p }}$$
for $x,y\in X$, $p\in \partial X$.
$\del$-hyperbolicity implies that for all $x,y\in X$ and $q\in \partial X$ we have
\begin{equation}\label{eqbuseq}
    |\beta_d(x,y;p)-(d(x,y)-2\groprod{x}{p}{y,d})|\leq 2\del
\end{equation} and 
$$|\beta_d(x,y;p)+\beta_d(y,x;p)|\leq 8\del.$$

In addition, for all $\ep>0$ satisfying $\ep\del<\frac{\log{2}}{6}$ there is a metric $\varrho=\varrho_\ep$ on $\partial X$ such that for all $p,q\in \partial X$
\begin{equation}\label{eqvismet}
    (3-2e^{3\ep\del})e^{-\ep\groprod{p}{q}{1,d}}\leq \varrho_\ep(p,q)\leq e^{-\ep\groprod{p}{q}{1,d}}.
\end{equation}
These are called \emph{visual metrics} on $\partial X$, and induce a canonical topology on $\partial X$ when $(X,d)$ is proper  \cite[Prop.~10, Ch.~7]{GhydelaHarpe1990book}.

\subsection{Hyperbolic groups}
A group $G$ is \emph{hyperbolic} if some (hence any) word metric with respect to a finite symmetric generating set is hyperbolic. $G$ is \emph{non-elementary} if it is not virtually abelian, in which case the Gromov boundary $\partial G$ is a well-defined infinite, compact metrizable space. 

Two pseudometrics $d_1$, $d_2$ on $G$ are $(\lam_1,\lam_2,\ep)$-\emph{quasi-isometric} if the identity $id:(G,d_1)\ra (G,d_2)$ is a $(\lam_1,\lam_2,\ep)$-quasi-isometry. Recall that $\calD(G)$ is the set of all hyperbolic and left-invariant pseudometrics on $G$, quasi-isometric to some word metric with respect to a finite generating set, and that $\scrD(G)$ is the quotient of $\calD(G)$ under the relation of being $(\lam^{-1},\lam,\ep)$-quasi-isometric for some $\lam>0,\ep\geq 0$.

For any $d\in \calD(G)$ and $n\geq 0$, we let $B_d(n)$ denote the closed ball of radius $n$ at the identity element, with respect to $d$. The \emph{critical exponent} of $d\in \calD(G)$ is then the quantity
$$h(d):=\lim_{n\to \infty}{\frac{\log{|B_d(n)|}}{n}},$$
well-defined by subadditivity. It is well-known that $h(d)$ is finite and positive, so that the rescaled metric $\hat{d}=h(d)\cdot d$ satisfies $h(\hat{d})=1$. Considering this, for $\rho\in \scrD(G)$ we define $$\hat\rho:=\corchete{d\in \rho \colon h(d)=1}.$$

\begin{defi}\label{defqcmeas}
Let $d\in \calD(G)$. A Borel probability measure $\nu$ on $\partial G$ is \emph{quasiconformal} for $d$ if $x\nu \ll \nu$ for all $x\in G$ and there exists a constant $C\geq 1$ such that for every $x\in G$ and $\nu$-almost every $p\in \partial G$ we have
\begin{equation}\label{eqqcmeasdef}
     C^{-1}e^{-h(d)\beta(x,1;p)}\leq \frac{dx\nu}{d\nu}(p)\leq Ce^{-h(d)\beta(x,1;p)},
\end{equation}
where $\beta=\beta_d$ is the Busemann function for $d$.
\end{defi}
It is clear that if $[d]=[d']$ in $\scrD(G)$, then $\nu$ is quasiconformal for $d$ if and only if it is quasiconformal for $d'$. Coornaert \cite{Coornaert1993PSmeasures} proved the following for arbitrary $d\in \calD(G) $:
\begin{itemize}
    \item $d$ admits a quasiconformal measure.
    \item If $\nu_1,\nu_2$ are quasiconformal measures for $d$, then they are \emph{equivalent}, in the sense that there is some $C \geq 1$ such that $C^{-1}\nu_2(A) \leq\nu_1(A) \leq C\nu_2(A)$ for any $A\subset \partial G$ Borel.
    \item Every quasiconformal measure $\nu$ for $d$ is \emph{ergodic}: if $A\subset \partial G$ is a $G$-invariant Borel subset, then either $\nu(A)=0$ or $\nu(A)=1$.
\end{itemize}

Moreover, the multiplicative constant in \eqref{eqqcmeasdef} can be chosen uniformly in the following sense: for all $\del\geq 0$ there is some $D_\del\geq 1$ such that any $\del$-hyperbolic pseudometric $d\in \calD(G)$ with $h(d)=1$ admits a quasiconformal measure $\nu$ such that for all $x\in G$ and $\nu$-almost every $p\in \partial G$: 
\begin{equation}\label{equnifqc}
    D_\del^{-1}e^{-\beta_d(x,1;p)}\leq \frac{dx\nu}{d\nu}(p) \leq D_\del e^{-\beta_d(x,1;p)}. 
\end{equation}
This uniformity is crucial in our proof of Theorem \ref{BMctns} (see Proposition \ref{convtoqc}).


\section{The Marked Length Spectrum}\label{secMLS}
In this section we prove Proposition \ref{Lambda=DilDil},  which characterizes the metric $\Del$ on $\scrD(G)$ in terms of marked length spectra, recovering the (symmetrized) Thurston's distance when restricted to Teichmüller space. Our main tool is Proposition \ref{thmDGLM}, which will also be used in later sections.

Given $d\in \calD(G)$ and $x\in G$, the \emph{translation length} of $x$ is given by the formula
\begin{equation*}
    \ell_d(x)=\lim_{n\to \infty}{\frac{d(x^n,1)}{n}},
\end{equation*}
well-defined by subadditivity. Recall that the \emph{marked length spectrum} of $d$ is the function 
$$\ell_d:G \ra \R, \hspace{2mm} x \mapsto \ell_d(x).$$
The translation length of $x\in G$ is positive if and only if $x$ is infinite order, and we have that $d,d'\in \calD(G)$ are roughly similar if and only if $\ell_d=\lam\ell_{d'}$ for some constant $\lam>0$ \cite[Thm.~1.2]{CantrellTanaka2021Manhattan}.

\begin{rmk}\label{torsionfree}
For any pseudometric $d\in \calD(G)$ and $x,y\in G$, we have $d(x,y)=d(y^{-1}x,1)\geq \ell_d(y^{-1}x)$. Since $\ell_d(y^{-1}x)>0$ when $y^{-1}x$ is infinite order, we deduce that the subgroup $\corchete{x\in G\colon d(x,1)=0}$ is torsion, hence it is a finite. In particular, if $G$ is torsion-free, then any pseudometric in $\calD(G)$ is a genuine metric.
\end{rmk}

The marked length spectrum characterizes metric structures \cite{Furman2002coarsenegcurv}, and indeed, we can partially recover a pseudometric in $\calD(G)$ from its marked length spectrum. This is the content of the following proposition, which is a generalization of \cite[Prop.~2.2.2]{DelzantGuichardLabourieMozes2011Displacing}.
\begin{prop}\label{thmDGLM} For any $d\in \calD(G)$ and $\ep\in (0,1)$, there is a finite subset $U\subset G$ and a constant $C \geq 0$ such that for all $x\in G$, 
\begin{equation}\label{eqapproxbyell}
    (1-\ep)d(x,1)\leq \max_{u\in U}{\ell_d(xu)}+C.
\end{equation}
\end{prop}
\begin{proof}
The proof is just an adaptation of the argument in \cite[Prop.~2.2.2]{DelzantGuichardLabourieMozes2011Displacing}, so we will only mention when modifications are needed. Assuming that $d$ is $\del$-hyperbolic, given $\ep\in (0,1)$ we say that an element $x\in G$ is $\ep$-\emph{almost cyclically reduced} if $\groprod{x}{x^{-1}}{1,d}\leq \frac{\ep}{2}d(x,1)-\del$ (note that \cite[Prop.~2.2.2]{DelzantGuichardLabourieMozes2011Displacing} only considers the case $\ep=2/3)$. 

Mimicking the proof of \cite[Lemma.~2.2.3]{DelzantGuichardLabourieMozes2011Displacing} we deduce that if $x$ is $\ep$-almost cyclically reduced then $(1-\ep)d(x,1)\leq \ell_d(x)$. Also, if $u,v \in G$ is a ping pong pair in the sense of \cite[Sec.~2.2.4]{DelzantGuichardLabourieMozes2011Displacing} then an adaptation of the proof of \cite[Lemma.~2.2.6]{DelzantGuichardLabourieMozes2011Displacing} implies that if $x\in G$ satisfies $d(x,1)\geq (1+\frac{2}{\ep})\max(d(u,1),d(v,1))+8\del/\ep$, then some of the three elements $x,xu, xv$ is $\ep$-almost cyclically reduced.

In consequence, the proposition follows with $C=(1-\ep)((1+\frac{2}{\ep})\max({d(u,1),d(v,1)})+8\del/\ep)$ and $U=\corchete{1,u,v}$.
 \end{proof}
The next definition follows \cite{CantrellTanaka2021Manhattan}: 
\begin{defi}
Given two pseudometrics $d,d'\in \calD(G)$, we define its \emph{(positive) dilation} by the formula
\begin{equation}\label{dilationeq}
    \Dil(d,d'):=\sup_{x}\frac{\ell_d(x)}{\ell_{d'}(x)},
\end{equation}
where $x$ runs over all infinite order elements of $G$.
\end{defi}
\begin{lemma}\label{lem.Dilqi}
For any $d, d'\in \calD(G)$ and $\ep\in (0,1)$ there is a constant $C'\geq 0$ such that 
\begin{equation}\label{eq.Dilqi}
   (1-\ep)\Dil(d',d)^{-1}d'(x,y)-C'\leq d(x,y)\leq  (1-\ep)^{-1}\Dil(d,d')d'(x,y)+C'
\end{equation}
for all $x,y\in G$.
\end{lemma}
\begin{proof}
First, we note that $\Dil(d',d)^{-1}\ell_{d'}(x)\leq \ell_d(x)\leq \Dil(d,d')\ell_{d'}(x)$ for all $x\in G$. Then, by Proposition \ref{thmDGLM} above, for any $\epsilon\in (0,1)$ we can find a finite set $U\subset G$ and some $C\geq 0$ such that equation \eqref{eqapproxbyell} holds for both $d$ and $d'$. In particular, for any $x,y \in G$ we have
\begin{align*}
    d(x,y) &\leq (1-\ep)^{-1}\max_{u\in U}{\ell_d(y^{-1}xu)}+C \\
    & \leq (1-\ep)^{-1}\Dil(d,d')\max_{u\in U}{\ell_{d'}(y^{-1}xu)}+C\\
    &  \leq (1-\ep)^{-1}\Dil(d,d')d'(x,y)+C+(1-\ep)^{-1}\Dil(d,d')\max_{u\in U}{d'(u,1)},
\end{align*}
and for the second inequality in \eqref{eq.Dilqi} we only require $C'\geq C+(1-\ep)^{-1}\Dil(d,d')\max_{u\in U}{d'(u,1)}$. The first inequality in \eqref{eq.Dilqi} is proven similarly.
\end{proof}
\begin{prop}\label{Lambda=DilDil}
For any $\rho,\rho'\in \scrD(G)$ and $d\in \rho, d'\in \rho'$ we have
$$\Lam(\rho,\rho')=\Dil(d,d')\Dil(d',d).$$
\end{prop}
\begin{proof}
If there is a $(\lam_1,\lam_2,\ep)$-quasi-isometry from  $d\in \rho$ to $d'\in \rho'$, then $\ell_{d'}/\lam_1\leq \ell_d\leq \lam_2\ell_{d'}$,  implying that $\Dil(d,d')\Dil(d',d)\leq \Lam(\rho,\rho')$. To prove the reverse inequality, by Lemma \ref{lem.Dilqi}, for any $\epsilon\in (0,1)$ we can find a constant $C'\geq 0$ such that $d$ and $d'$ are $((1-\ep)^{-1}\Dil(d',d),(1-\ep)^{-1}\Dil(d,d'),C')-$quasi-isometric. As this holds for any $\ep\in (0,1)$, we deduce that $\Lam(\rho,\rho')\leq \Dil(d',d)\Dil(d,d')$.
\end{proof}
In some cases, we will choose representatives of metric structures with critical exponent 1. Under this assumption we have:
\begin{lemma}\label{normalizedcriticalexponent}
For all $d,d'\in \calD(G)$, if $h(d)=h(d')=1$ then $\Dil(d,d')\geq 1$. 
\end{lemma}
\begin{proof}
By applying Lemma \ref{lem.Dilqi}, for all $\ep \in (0,1)$ we can find some $C'\geq 0$ such that $d(x,1)\leq (1-\ep)^{-1}\Dil(d,d')d'(x,1)+C'$ for any $x\in G$. In particular we have $B_{d'}(n)\subset B_d((1-\ep)^{-1}\Dil(d,d')n+C')$ for all $n\geq 0$, implying that 
\begin{align*}
    1=\lim_{n\to \infty}{\frac{\log{|B_{d'}(n)|}}{n}} &\leq \lim_{n\to \infty}{\frac{\log{|B_d((1-\ep)^{-1}\Dil(d,d')n+C')|}}{n}}\\ & \leq (1-\ep)^{-1}\Dil(d,d')\lim_{n\to \infty}{\frac{\log{|B_d(n)|}}{n}}=(1-\ep)^{-1}\Dil(d,d').
\end{align*}
This holds for all $\ep\in (0,1)$, concluding the desired inequality.
\end{proof}
\begin{rmk}
Proposition \ref{Lambda=DilDil} and Lemma \ref{normalizedcriticalexponent} can also be deduced from properties of the Manhattan curve \cite[Cor.~3.3]{CantrellTanaka2021Manhattan}.
\end{rmk}
\begin{rmk}\label{Thurstonmetric}
Given $\rho,\rho' \in \scrD(G)$ one can define $$\Del^+(\rho,\rho')=\log\Dil(d,d'),$$ where $d\in \hat{\rho}$ and $d'\in \hat\rho'$. Lemma \ref{normalizedcriticalexponent} then tells us that $\Del^+$ is nonnegative, so that it defines an asymmetric distance on $\scrD(G)$. In the case $G$ is a surface group and $\rho,\rho'\in \scrT(G)$, since cocompact lattices in $\Hy^2$ have critical exponent 1, from \eqref{dilationeq} we get that $\Del^+(\rho,\rho')$ is \emph{equal} to the asymmetric Thurston's distance of $\rho$ and $\rho'$ \cite{Thurston1998Minstretchmaps}. The content of Proposition \ref{Lambda=DilDil} is then that $\Del$ is twice the symmetrization of $\Del^+$.
\end{rmk}


\section{$\scrD(G)$ is contractible and unbounded}\label{seccontractibleunbounded}
In this section we prove Theorem \ref{unboundedcontractible}, which splits into Propositions \ref{contractible} and \ref{unbounded}. We always assume that $G$ is hyperbolic and non-elementary. We start with a lemma.
\begin{lemma}\label{sumhypishyp}
If $d_1,d_2\in \calD(G)$, then $d_1+d_2\in \calD(G)$.
\end{lemma}
\begin{proof}
As $d_1$ and $d_2$ are left-invariant and quasi-isometric to a word metric, it is clear that $d:=d_1+d_2$ is a pseudometric on $G$,  also left-invariant and quasi-isometric to a word metric, so that we are left to prove hyperbolicity of $d$. By \cite[Thm.~1.10]{BlachereHaissinskyMathieu2011HarmvsQconf} there exists a set $\calG=\corchete{g_{x,y}:I_{x,y}\ra (G,d_1) \colon x,y\in G}$ of uniform quasigeodesics with each $g_{x,y}$ joining $x$ and $y$, and such that each $g_{x,y}$ is a $\rho$-quasiruler for $\rho$ uniform in $\calG$, in the sense of \cite[Sec.~1.7]{BlachereHaissinskyMathieu2011HarmvsQconf}.  Since $d_1$ and $d_2$ are quasi-isometric, by Proposition \ref{strongqi} the set $\calG$ is of uniform quasirulers for $d$, and hence $d$ is hyperbolic by \cite[Thm.~1.10]{BlachereHaissinskyMathieu2011HarmvsQconf}. \end{proof}

\begin{prop}\label{contractible}
$\scrD(G)$ is contractible.
\end{prop}
\begin{proof}
Fix $\rho_0\in \scrD(G)$ and a pseudometric $d_0\in \hat{\rho}_0$. We will construct a map $H:\scrD(G)\times [0,1]\ra \scrD(G)$ which is constant at $t=0$ and the identity at $t=1$ as follows. For each $\rho\in \scrD(G)$ and $t\in [0,1]$, consider some $d\in \hat\rho$, and let $d_t:=td+(1-t)d_0$. By Lemma \ref{sumhypishyp}, $d_t\in \calD(G)$, and we have a metric structure $[d_t]\in \scrD(G)$. Note that if $d$ and $d'$ are roughly isometric, $d_t$ and $d'_t$ are also roughly isometric, so that $H(\rho,t):=[d_t]$ is well-defined. 

It easily follows that $H(\rho,0)=\rho_0$ and $H(\rho,1)=\rho$ for all $\rho$, so we are only left to check that $H$ is continuous. To do this, let $\rho \in \scrD(G)$ and consider $0<t<1$. For any $\rho' \in \scrD(G), d'\in \hat\rho'$ and $0<s<1$ we have
\begin{align*}
    d_t=td+(1-t)d_0\leq t\Lam(\rho,\rho')d'+(1-t)d_0+tC\leq \max\left(\frac{t\Lam(\rho,\rho')}{s},\frac{1-t}{1-s}\right)d'_s+tC,
\end{align*}
and 
\begin{align*}
    d_t\geq t\Lam(\rho,\rho')^{-1}d'+(1-t)d_0+tC\geq \min\left(\frac{t}{s\Lam(\rho,\rho')},\frac{1-t}{1-s}\right)d'_s-tC,
\end{align*}
for some $C$ independent of $s$ and $t$. This implies that if $\Del(\rho',\rho) \to 0$ and $s \to t$, then $\Del(H(\rho',s),H(\rho,t))\to 0$. The cases where $t=0$ or $t=1$ are similar, and are left to the reader.
\end{proof}

\begin{prop}\label{unbounded}
$\scrD(G)$ is unbounded.
\end{prop}
\begin{proof}
Given $x\in G$ an infinite order element, it is enough to construct a sequence $(\rho_n)_n\subset \scrD(G)$ and $d_n\in \hat\rho_n$ such that $\ell_{d_n}(x)\to 0$ as $n\to \infty$.
To do this, fix $S$ a finite generating set for $G$, and for $x$ as above consider the sequence of sets $S_n=S\cup \corchete{x^n,x^{-n}}$ and define $\rho_n$ as the rough similarity class of $d_{S_n}$. If $d_n=h(d_{S_n})d_{S_n}$, for all $n$ we have $h(d_n)=1$ and 
$\ell_{d_n}(x)=h(d_{S_n})\ell_{d_{S_n}}(x)\leq \frac{h(d_{S_n})|x^n|_{S_n}}{n}\leq \log(|S|+1)/n$, and hence $\ell_{d_n}(x)$ tends to $0$ and $\Del(\rho_n,\rho_1)$ tends to infinity as $n$ tends to infinity.
\end{proof}


\section{Properness of $\scrD_\del(G)$ and $\scrD_{\del,\al}(G)$}\label{secproperness}
In this section we prove Theorem \ref{proper}, divided into Propositions \ref{complete}, \ref{propersecond} and \ref{boundimpliesroughgeod}. We start with the simple fact that any metric structure can be approximated by word metrics. Recall that for $S$ a finite symmetric generating set for $G$, the word metric on $G$ with respect to $S$ is given by
$$d_S(x,y)=|x^{-1}y|_S,$$
where $|w|_S$ denotes the minimal $k$ such that $w=s_1\cdots s_k$ with $s_i\in S$, and the convention that $|1|_S=0$. The next lemma is a variation of \cite[Lem.~4.6]{BessonCourtoisGallotSambusetti2021Finitenesshyp}.
\begin{lemma}\label{worddense}
$d\in \calD(G)$ be an $\al$-roughly geodesic pseudometric. Then for all $n>\al+1$ such that $S_n=B_d(n)$ generates $G$ we have
\begin{equation*}
(n-1-\al)d_{S_n}(x,y)-(n-1)\leq d(x,y)\leq nd_{S_n}(x,y)
\end{equation*}
for all $x,y\in G$.
\end{lemma}
\begin{proof}
Let $n> \al+1$ be such that $S_n$ generates $G$. If $x,y\in G$ are such that $d(x,y)_{S_n}=k>0$, then $x^{-1}y=x_1\cdots x_k$ with $x_i\in S_n$, and hence $d(x,y)\leq d(x_1,1)+\cdots d(x_k,1)\leq nk=nd_{S_n}(x,y)$. This proves the second inequality. 

For the first inequality, let $x\in G$ and consider an $\al$-rough geodesic sequence $1=x_0,\dots,x_l=x$, so that $|i-j|-\al\leq d(x_i,x_j)\leq |i-j|+\al$ for all $i,j$. Let $m$ be the greatest integer with $m+\al\leq n$ (note that $m$ is positive), and say $l=mk+r$ with $k,r$ integers such that $0\leq r<m$. If we define $y_j=x^{-1}_{m(j-1)}x_{mj}$ for $1\leq j \leq k$ and $y_{k+1}=x^{-1}_{mk}x$, then $x=y_1\cdots y_{k+1}$, and each $y_j$ lies in $B_d(n)$, so that $|x|_{S_n}\leq k+1$. This implies
\begin{align*}
    d(x,1)  \geq l-\al\geq mk-\al &\geq (n-1-\al)(d_{S_n}(x,1)-1)-\al \\ & =(n-1-\al)d_{S_n}(x,1)-(n-1),
\end{align*}
and the first inequality then follows from left-invariance of $d$ and $d_{S_n}$.
\end{proof}
As there are only countably many word metrics, the previous lemma already implies that ($\scrD(G),\Del)$ is separable. Another consequence is that when $G$ is torsion-free, for an arbitrary metric $d\in \calD(G)$, its translation length $\ell_d$ can be continuously extended to $\Curr(G)$. This was proved by Erlandsson, Parlier and Souto in case $d$ comes from a geometric action on a geodesic metric space \cite[Thm.~1.5]{ErlandssonParlierSouto2020Stablelengthhyp}. 

\begin{coro}\label{mlsctns}
If $G$ is torsion-free, then for any $d\in \calD(G)$, its translation length $\ell_d$ extends uniquely to a continuous homogeneous function $\ell_d:\Curr(G)\ra \R$.
\end{coro}
\begin{proof}
By \cite[Thm.~1.5]{ErlandssonParlierSouto2020Stablelengthhyp}, the result follows if $d$ is a word metric. For arbitrary $d\in \calD(G)$, let $\lambda_nd_n\in \calD(G)$ be a sequence of word metrics on $G$ such that $\Lambda_n:=\Lambda(d,d_n)\to 1$ as $n\to \infty$, normalized so that $h(\lambda_nd_n)=h(d)$. By Lemma \ref{normalizedcriticalexponent} we have ${\Lambda_n}^{-1}\ell_{\lambda_nd_n}\leq \ell_d\leq \Lambda_n\ell_{\lambda_nd_n}$ for all $n$, and hence the functions $\log\ell_{\lambda_nd_n}:\Curr(G)\ra \R$ converge uniformly to some function $f$. Continuity of each $\ell_{d_n}$ then guarantees the continuity of $e^f$, which is the desired homogeneous extension of $\ell_d$ to $\Curr(G)$. Uniqueness is deduced from the density of rational currents \cite[Thm.~7]{Bonahon1991currentsnegcurv}.
\end{proof}

\begin{rmk}
The corollary above implies that when $G$ is torsion-free, there exists a left-invariant metric $d_0$ on $G$, quasi-isometric to a word metric, for which $\ell_{d_0}$ does not extend continuously to $\Curr(G)$ (in particular, $d_0$ is not hyperbolic, cf.~\cite[Prop.~A.11]{BlachereHaissinskyMathieu2011HarmvsQconf}). Indeed, let $\phi:G \ra \ov{G}$ be a non-elementary hyperbolic quotient of $G$ with infinite kernel, which can be constructed from group theoretic Dehn filling \cite{GrovesManning2008DehnFilling,Osin2007Perfilling}. Given any metric $d\in \calD(\ov{G})$, consider the pseudometric  $\ov{d}(x,y)=d(\phi(x),\phi(y))$ on $G$. The same  argument as in the proof of \cite[Prop.~11]{Bonahon1991currentsnegcurv} implies that the marked length spectrum $\ell_{\ov{d}}$ does not extend continuously to $\Curr(G)$. Therefore, for $d_1\in \calD(G)$, the metric $d_0=\ov{d}+d_1$ is left-invariant and quasi-isometric to the word metric, and by Corollary \ref{mlsctns} its marked length spectrum $\ell_{d_0}=\ell_{\ov{d}}+\ell_{d_1}$ does not extend continuously to $\Curr(G)$. 
\end{rmk}

To prove the properness of $\scrD_\del(G)$ and $\scrD_{\del,\al}(G)$, we will need to find a convergent subsequence for a sequence of hyperbolic pseudometrics representing a bounded sequence of metric structures. To guarantee existence and good properties of the limit pseudometric, we need some way to choose well-behaved representatives for the metric structures. The existence of such good pseudometrics is ensured by the following Bochi-type inequality for isometries of hyperbolic spaces due to Breuillard and Fujiwara \cite[Thm.~1.4]{BreuillardFujiwara2021JSRnonpos}. They proved this result when $(X,d)$ is metric and geodesic, but their proof can be adapted to the case when $(X,d)$ is just pseudometric and $\al$-rough geodesic. Details are left to the reader. 

\begin{thm}[Breuillard-Fujiwara]\label{thmbrefuj}
For all $\del,\al\geq 0$ there exists a positive constant $K(\del,\al)$ such that if $(X,d)$ is a $\del$-hyperbolic, $\al$-rough geodesic pseudometric space and $S\subset \Isom(X)$ is a finite set of isometries, then
\begin{equation*}
\inf_{x\in X}{\max_{s\in S}{d(sx,x)}}\leq K(\del,\al)+\frac{1}{2}{\max_{s_1,s_2\in S}{\ell_d(s_1s_2)}}.
\end{equation*}
\end{thm}

Given $\delta,\al\geq 0$, recall that $\scrD_{\del,\al}(G)\subset \scrD(G)$ is the subset of all such metric structures represented by a $\del$-hyperbolic and $\al$-rough geodesic pseudometric with critical exponent 1. The next lemma will be used to find good representatives for metric structures.

\begin{lemma}\label{lem.unifquasiisom}
Let $S\subset G$ be a finite symmetric generating set, and $B\subset \scrD_{\del,\al}(G)$ a bounded subset. Then there exists some $C\geq 1$ depending only on $B$ and $S$, such that for every $\rho\in B$ there exists a $\del$-hyperbolic and $\al$-rough geodesic pseudometric $d_\rho\in \hat\rho$ satisfying
\begin{equation}\label{eq.unifquasigeod}
    C^{-1}d_S(x,y)-C\leq d_\rho(x,y)\leq Cd_S(x,y)
\end{equation}
for all $x,y \in G$.
\end{lemma}

\begin{proof}
Since $B$ is bounded, there exists some $\Lam \geq 1$ such that for all $\rho\in B$, $d\in \hat{\rho}$ and $x\in G$, we have 
\begin{equation}\label{eq.quasiisomell}
    \Lam^{-1}\ell_S(x)\leq \ell_d(x)\leq \Lam\ell_S(x),
\end{equation}
where $\ell_S=\ell_{d_S}$. Now, for each $\rho \in B$, let $d'_\rho\in \hat{\rho}$ be a $\del$-hyperbolic and $\al$-rough geodesic pseudometric. We want to modify these pseudometrics so that they satisfy \eqref{eq.unifquasigeod} for some $C$ independent of $\rho$. To do this, for any $\rho \in B$ we apply Theorem \ref{thmbrefuj} to the finite set $S$ and the pseudometric $d'_\rho$ to find a point $z_\rho\in G$ such that 
\begin{equation}\label{eq.boundmaxS}
    \max_{s\in S}{d'_\rho(sz_\rho,z_\rho)}\leq K(\del,\al)+\frac{1}{2}\max_{s_1,s_2\in S}{\ell_{d'_\rho}(s_1s_2)} \leq K(\del,\al)+\frac{\Lam}{2}\max_{s_1,s_2\in S}{\ell_S(s_1s_2)},
\end{equation}
where in the last inequality we used \eqref{eq.quasiisomell}. We define $C_1$ as the last term in \eqref{eq.boundmaxS}, which is independent of $\rho\in B$. In this way, the pseudometric $d_\rho(x,y):=d'_\rho(xz_\rho,yz_\rho)$ is also left-invariant, $\del$-hyperbolic and $\al$-rough geodesic, and for each $x,y\in G$ we have 
\begin{equation}\label{eq.upperboundqi}
    d_\rho(x,y)\leq C_1d_S(x,y).
\end{equation}
For the other inequality in \eqref{eq.unifquasigeod}, we apply Proposition \ref{thmDGLM} to $d_S$ and $\ep=1/2$, to find a finite set $U\subset G$ and a constant $C_2\geq 0$ such that 
\begin{equation*}
d_S(x,y)\leq 2\max_{u\in U}{\ell_S(y^{-1}xu)}+C_2
\end{equation*}
for all $x,y\in G$. This last inequality together with \eqref{eq.quasiisomell} and \eqref{eq.upperboundqi} imply that for $\rho \in B$ and all $x,y\in G$
\begin{align*}
     d_\rho(x,y) & \geq \max_{u\in U}{d_\rho(y^{-1}xu,1)} -\max_{u\in U}{d_\rho(u,1)} 
     \\
     & \geq 
     \max_{u\in U}\ell_{d_\rho}(y^{-1}xu)-C_1\max_{u\in U}{d_S(u,1)} \\ 
     & = \max_{u\in U}\ell_{d'_\rho}(y^{-1}xu)-C_1\max_{u\in U}{d_S(u,1)} \\
     & \geq \Lam^{-1}\max_{u\in U}\ell_{S}(y^{-1}xu)-C_1\max_{u\in U}{d_S(u,1)} \\
     & \geq \frac{\Lam^{-1}}{2}(d_S(x,y)-C_2)-C_1\max_{u\in U}{d_S(u,1)} \\
     & = \frac{\Lam^{-1}}{2}d_S(x,y)-\left(\frac{\Lam^{-1}C_2}{2}+C_1\max_{u\in U}{d_S(u,1)}\right),
\end{align*}
and inequality \eqref{eq.unifquasigeod} holds for all $\rho$ in $B$, with $C=\max(C_1,2\Lam,\frac{\Lam^{-1}C_2}{2}+C_1\max_{u\in U}{d_S(u,1)})$.
\end{proof}

Now we begin the proof of Theorem \ref{proper}. Our first step is to prove the completeness of $\scrD_{\del,\al}(G)$.

\begin{prop}\label{complete}
For all $\del,\al\geq 0$ the set $\scrD_{\del,\al}(G)$ is either empty, or a complete subspace of $\scrD(G)$.
\end{prop}

\begin{proof}
Assume that $\scrD_{\del,\al}(G)$ is non-empty and let $\rho_n\in \scrD_{\del,\al}(G)$ be a Cauchy sequence of metric structures, for which we expect to converge to some $\rho_\infty\in \scrD_{\del,\al}(G)$. 

The sequence $\rho_n$ is bounded, and hence by Lemma \ref{lem.unifquasiisom} there exists a constant $C\geq 1$ independent of $n$, and $\del$-hyperbolic and $\al$-rough geodesic pseudometrics $d_n\in \hat\rho_n$ such that for each $x,y\in G$ we have 
\begin{equation}\label{upperboundqi}
   C^{-1}d_S(x,y)-C\leq  d_n(x,y)\leq Cd_S(x,y).
\end{equation}
This last inequality allows us to find a subsequence $(n_k)_k$ such that for each $x,y\in G$, the sequence $d_{n_k}(x,y)$ has a limit, so that for $x,y\in G$ we define $d_\infty(x,y):=\lim_k{d_{n_k}(x,y)}$. To conclude the result, there are some claims to verify:

Claim 1: $d_\infty\in \calD(G)$. 

Clearly $d_\infty$ is left-invariant, symmetric and satisfies the triangle inequality, so it is a pseudometric on $G$. Since each $d_n$ is $\del$-hyperbolic, $d_\infty$ is also $\del$-hyperbolic, and it is quasi-isometric to $d_S$ as a consequence of \eqref{upperboundqi}, so that $d_\infty\in \calD(G)$.

Claim 2: The sequence $\rho_n$ converges to $\rho_\infty:=[d_{\infty}]$. 

We first prove that $\ell_{d_{n_k}}(x)$ converges to $\ell_{d_\infty}(x)$ for all $x\in G$. Let $x\in G$, for which we have
\begin{align*}
    \ell_{d_\infty}(x)=\lim_{m\to \infty}{\frac{d_\infty(x^m,1)}{m}}=\lim_{m\to \infty}{\frac{\lim_k{d_{n_k}(x^m,1)}}{m}} \geq \limsup_k{\ell_{d_{n_k}(x)}},
\end{align*}
where in the last inequality we used that $\ell_{d}(x)=\inf_{m \geq 1}{\frac{d(x^m,1)}{m}}$ for all $d\in \calD(G)$ and $x\in G$. For the inequality $\ell_{d_\infty}(x)\leq \liminf_k{\ell_{d_{n_k}(x)}}$ we use \cite[Thm.~1.1]{OregonReyes18propsetsisom}, which implies that $$\ell_d(x)=\sup_{m\geq 1}{\left(\frac{d(x^{2m},1)-d(x^m,1)-2\del}{m}\right)}$$ for each $\del$-hyperbolic pseudometric $d\in \calD(G)$. From this we get 
\begin{align*}
    \ell_{d_\infty}(x) &=\lim_{m\to \infty}{\left(\frac{d_\infty(x^{2m},1)-d_\infty(x^m,1)-2\del}{m}\right)} \\&=\lim_{m\to \infty}{\left(\frac{\lim_k(d_{n_k}(x^{2m},1)-d_{n_k}(x^m,1)-2\del)}{m}\right)} \leq \liminf_k{\ell_{d_{n_k}(x)}}.
\end{align*}
This convergence implies that 
for all $n$ and $x\in G$
\begin{equation}\label{eqconvergence}
    (\limsup_{k\to \infty}{\Dil(d_{n},d_{n_k})})^{-1}\ell_{d_n}(x) \leq \ell_{d_\infty}(x)\leq \limsup_{k\to \infty}{\Dil(d_{n_k},d_n)}\ell_{d_n}(x).
\end{equation}
As each $d_n$ has critical exponent 1, by Proposition \ref{Lambda=DilDil} and Lemma \ref{normalizedcriticalexponent} we have that \begin{equation}\label{eq.Dilmn}
\max(\Dil(d_{n},d_{m}),\Dil(d_{m},d_{n}))\leq \Lam(\rho_{m},\rho_n)
\end{equation}
for all $m,n$. Also, since $(\rho_n)_n$ is Cauchy, the sequence $\Lam_n:=\limsup_{k\to \infty}{\Lam(\rho_{n_k},\rho_n)}$ tends to 1, and hence $\rho_n$ converges to $\rho_\infty$.

Claim 3: $\rho_\infty\in \scrD_{\del,\al}(G)$.

We already proved that $d_\infty$ is $\del$-hyperbolic, so we now show that $h(d_\infty)=1$. Indeed, for each $d\in \calD(G)$, the quantity $h(d)$ is also the critical exponent of
\begin{equation*}
    b \mapsto \sum_{x \in \mathrm{conj}'}{e^{-b\ell_d(x)}},
\end{equation*}
where $\mathrm{conj}'$ is a complete set of representatives for the conjugacy classes of infinite order elements of $G$ (see e.g. \cite[Prop.~3.1]{CantrellTanaka2021Manhattan}, the case $a=0$). Applying this to $d_\infty$ and each $d_n$, and using \eqref{eqconvergence} and \eqref{eq.Dilmn}, we get that 
$$\Lam_n^{-1}h(d_n)\leq h(d_\infty)\leq \Lam_n h(d_n)$$
for all $n$. Since $h(d_n)=1$ for each $n$ and $\Lam_n$ tends to 1, we deduce $h(d_\infty)=1$.

We are only left to show that $d_\infty$ is $\al$-rough geodesic. To do this, fix $x\in G$, and for each $k$, let $1=x_{0,k},x_{1,k},\dots,x_{r_k,k}=x$ be an $\al$-rough geodesic for $d_{n_k}$ joining $1$ and $x$. By \eqref{upperboundqi} we have $r_k\leq \al+d_{n_k}(x,1)\leq \al+C|x|_S$, so that the sequence $(r_k)_k$ is bounded, and after replacing $n_k$ by a further subsequence and reindexing, we can assume that $r_k=r$ for all $k$. From this, it is enough to prove that for all $0\leq j \leq r$, the sequence $(x_{j,k})_k$ is bounded in $d_{\infty}$, since in that case, and after extracting a subsequence, the constant sequence of points $x_j=x_{j,k}$ will define an $\al$-rough geodesic for $d_\infty$ joining $1$ and $x$. 

To prove this, note that for $0\leq j \leq r$ and any $k$ we have $d_{n_k}(x_{j,k},1)\leq j+\al$, so it is enough to show that for all $y\in G$ and $k$ we have 
$$d_\infty(y,1)\leq A'd_{n_k}(y,1)+B',$$
for some constants $A',B'$ independent of $y$ and $k$.
For this we apply Proposition \ref{thmDGLM} to $d_{\infty}$ and $\ep=1/2$, to get a finite set $V\subset G$ and a constant $E\geq 0$ such that $d_\infty(y,1)\leq 2\max_{v\in V}{\ell_{d_\infty}(vy)}+E$ for all $y$. Therefore, by this and \eqref{upperboundqi} we obtain
\begin{align*}
    d_\infty(y,1) &\leq 2\max_{v\in V}{\ell_{d_\infty}(yv)}+E \\ & \leq 2\Lam(\rho_\infty,\rho_{n_k})\max_{v\in V}{\ell_{d_{n_k}}(yv)}+E \\ &\leq 2\Lam(\rho_\infty,\rho_{n_k})d_{n_k}(y,1)+E+2\Lam(\rho_\infty,\rho_{n_k})\max_{v\in V}{d_{n_k}(v,1)} \\ &\leq 2\Lam(\rho_\infty,\rho_{n_k})d_{n_k}(y,1)+E+2C\Lam(\rho_\infty,\rho_{n_k})\max_{v\in V}{|v|_S},
\end{align*}
and the conclusion follows since the sequence $k\mapsto \Lam(\rho_\infty,\rho_{n_k})$ is bounded. This completes the proof of Claim 3 and the proposition.
\end{proof}
 
\begin{prop}\label{propersecond}
For each $\del,\al\geq 0$, the set $\scrD_{\del,\al}(G)$ is either empty or a proper subspace of $\scrD(G)$.
\end{prop}
\begin{proof}
Assume $\scrD_{\del,\al}(G)$ is non-empty. Then it is complete by Proposition \ref{complete}, so it is enough to show that for any bounded set $B\subset \scrD_{\del,\al}(G)$ and any $\ep>0$, $B$ can be covered by finitely many subsets of diameter at most $\ep$.

As $B$ is bounded and contained in $\scrD_{\del,\al}(G)$, by Lemma \ref{lem.unifquasiisom} we can find some $C\geq 1$ such that for every $\rho \in B$ there is some $\del$-hyperbolic and $\al$-rough geodesic pseudometric $d_\rho \in \hat\rho$ satisfying
\begin{equation}\label{qiwordB}
    C^{-1}d_S(x,y)_S-C\leq d_\rho(x,y) \leq Cd_S(x,y)
\end{equation}
for all $x,y\in G$.

Now, for every $\rho\in B$ and $n\geq 0$, let $S_{n,\rho}:=B_{d_{\rho}}(n)$. By \eqref{qiwordB}, for every $\rho \in B$ we have $S\subset S_{C,\rho}$, so that $S_{n,\rho}$ generates $G$ for all $n\geq C$. Therefore, by Lemma \ref{worddense} we have that for all $\rho \in B$, $n\geq \max(C,\al+2)$ and $x\in G$:
\begin{equation*}
(n-1-\al)d_{S_{n,\rho}}(x,1)-(n-1)\leq d_\rho(x,1)\leq nd_{S_{n_\rho}}(x,1).
\end{equation*}
Given $\ep>0$, let $n_0\geq \max(C,\al+2)$ be such that $\frac{n_0}{n_0-1-\al}<e^{\ep/2}$, and let $\calA\subset \scrD(G)$ be the set of all the metric structures $[d_T]$, with $T$ a finite generating set contained in $B_{S}(Cn_0+C^2)$. Note that $\calA$ is finite.

By \eqref{qiwordB} we obtain that $[d_{S_{n_0,\rho}}]\in \calA$  for every $\rho \in B$, and that  $\Del(\rho,[d_{S_{n_0,\rho}}])\leq \log(\frac{n_0}{n_0-1-\al})<\ep/2$, so that $B$ is covered by the finite collection of balls $ B_\Del(\sigma,\ep/2)$ with $\sigma \in \calA$.
\end{proof}

To finish the proof of Theorem \ref{proper}, we are left to show that $\scrD_{\del}(G)$ is either empty or proper for any $\del\geq 0$. In virtue of Proposition \ref{propersecond}, it is enough to prove that bounded subsets of $\scrD_{\del}(G)$ are contained in $\scrD_{\del,\al}(G)$ for some $\al\geq 0$. The following lemma is an adaptation of Lemma \ref{lem.unifquasiisom} for bounded subsets of $\scrD_\del(G)$.

\begin{lemma}\label{lem.unifquasiisomJUSTDEL}
Let $S\subset G$ be a finite symmetric generating set, and $B\subset \scrD_{\del}(G)$ a bounded subset. Then there exists some $C\geq 1$ depending only on $B$ and $S$, such that for every $\rho\in B$ there exists a $\del$-hyperbolic geodesic pseudometric $d_\rho\in \hat\rho$ satisfying
\begin{equation*}
    C^{-1}d_S(x,y)-C\leq d_\rho(x,y)\leq Cd_S(x,y)
\end{equation*}
for all $x,y \in G$.
\end{lemma}

\begin{proof}[Sketch of proof]
For each $\rho\in B$, let $d'_\rho\in \hat\rho$ be a $\del$-hyperbolic pseudometric, and by Proposition \ref{prop.injhull} consider its injective envelope $(G,d'_\rho)\xhookrightarrow{i_\rho} (X_\rho,\hat{d}_\rho)$ which is $\del$-hyperbolic and geodesic. By this same proposition, $G$ also acts isometrically on $(X_\rho,\hat{d}_\rho)$, and the isometric map $i_\rho$ is $G$-equivariant.

As in the proof of Lemma \ref{lem.unifquasiisom}, we can use that $B$ is bounded and apply Theorem \ref{thmbrefuj} to the set $S$ and each $(X_\rho,\hat{d}_\rho)$, to find a constant $C \geq 1$ such that for any $\rho\in B$ there is a point $z_\rho\in X_\rho$ satisfying
\begin{equation*}
    C^{-1}d_S(x,y)-C\leq \hat{d}_\rho(xz_\rho,yz_\rho)\leq Cd_S(x,y)
\end{equation*}
for all $x,y\in G$. The proof is completed by taking $d_\rho(x,y):=\hat{d}_\rho(xz_\rho,yz_\rho)$ for $x,y\in G$, details are left to the reader.
\end{proof}

\begin{prop}\label{boundimpliesroughgeod}
For any $\del\geq 0$, if $B\subset \scrD_\del(G)$ is bounded, then $B\subset \scrD_{\del,\al}(G)$ for some $\al$.
\end{prop}

\begin{proof}
Let $S\subset G$ be a finite symmetric generating set. We apply Lemma \ref{lem.unifquasiisomJUSTDEL} to $S$ and $B$ to find a constant $C \geq 1$ such that for any $\rho\in B$ there is a $\del$-hyperbolic pseudometric $d_\rho\in \hat\rho$ satisfying
\begin{equation*}
    C^{-1}d_S(x,y)-C\leq d_\rho(x,y)\leq Cd_S(x,y)
\end{equation*}
for all $x,y\in G$. In particular, for each $\rho\in B$, any pair of points in $(G,d_\rho)$ can be joined by a (discrete) $(C,C,C)$-quasigeodesic. Therefore, by Lemma \ref{lem.roughgeodfromhyp} applied to $(G,d_\rho)$ we deduce that there is some $\al\geq 0$ such that for every $\rho\in B$, the pseudometric $d_\rho$ is $\del$-hyperbolic and $\al$-rough geodesic, which concludes the proof of the proposition. \end{proof}


\section{The action of $\Out(G)$ and cocompactness}\label{secOutG}
In this section we prove Theorems \ref{Outmetricproper} and \ref{cocompactness}.
\begin{proof}[Proof of Theorem \ref{Outmetricproper}]
To prove the assertion, suppose by contradiction that there exists some $\rho \in \scrD(G)$, some $R>0$ and infinitely many $\psi\in \Out(G)$ such that $\Del(\rho, \psi(\rho))\leq R.$
Consider $d\in \hat\rho$ and let $\calA\subset \Aut(G)$ be a complete set of representatives of the elements $\psi\in \Out(G)$ such that $\ell_{\psi(d)}\leq e^{R}\ell_d$, which is infinite by assumption.

Let $S\subset G$ be a finite symmetric generating set, and for $\psi\in \calA$ consider the quantity $\lam_\psi=\inf_{x\in G}{\max_{s\in S}{d(x,\psi^{-1}(s)x)}}$. As $\calA$ is infinite, the set $\corchete{\lam_\psi}_{\psi \in \calA}$ is unbounded, see e.g. \cite[p.~338]{Paulin1991OuthypRtrees}. On the other hand, since each pseudometric $\psi(d)$ is $\del$-hyperbolic and $\al$-rough geodesic for some uniform $\del$ and $\al$, by Theorem \ref{thmbrefuj} there exists a constant $K>0$ such that for all $\psi\in \calA$ we have
\begin{align*}
    \lam_\psi \leq K+\frac{1}{2}\max_{s_1,s_2\in S}{\ell_{\psi(d)}(s_1s_2)}\leq K+\frac{e^R}{2}\max_{s_1,s_2\in S}{\ell_{d}(s_1s_2)}.
\end{align*}
The last term does not depend on $\psi$, implying the desired contradiction.
\end{proof}

We will need another result in the proof of Theorem \ref{cocompactness}, for which we need some notation. A \emph{marked group} is a pair $(G,S)$ where $G$ is a group and $S\subset G$ is a finite symmetric generating subset. If $G$ is hyperbolic, we say that $(G,S)$ is $\delta$-hyperbolic if the word metric $d_S$ is $\del$-hyperbolic. Similarly, the critical exponent of $(G,S)$ is $h(G,S)=h(d_S)$.  

An \emph{isomorphism of marked groups} $\psi:(G,S)\ra (G',S')$ is an isomorphism $\psi:G\ra G'$ such that $\psi(S)=S'$, and two marked groups are \emph{isometrically isomorphic} if there exists an isomorphism of marked groups between them. Note that the isomorphism of marked groups induces an isometry $(G,d_S) \ra (G',d_{S'})$, and when $G=G'$, the corresponding element in $\Out(G)$ induced by $\psi$ maps $[d_S]$ to $[d_{S'}]$.

The following finiteness result is a particular case of a theorem of Besson-Courtois-Gallot-Sambusetti \cite[Thm.~1.4]{BessonCourtoisGallotSambusetti2021Finitenesshyp}.

\begin{thm}[Besson-Courtois-Gallot-Sambusetti]\label{thmBCGS}
Let $G$ be a non-elementary, torsion-free hyperbolic group. For any $\del,H\geq 0$ there are only finitely many marked groups $(G,S)$ (up to isometric isomorphism) such that $(G,S)$ is $\del$-hyperbolic and $h(G,S)\leq H$.   
\end{thm}

\begin{proof}[Proof of Theorem \ref{cocompactness}]
Assume $\scrD_{\del,\al}(G)$ is non-empty, and let $
\rho\in \scrD_{\del,\al}(G)$ and $d_\rho\in \hat\rho$ be a $\del$-hyperbolic and $\al$-rough geodesic pseudometric representative. For this metric we have that $S_\rho:=B_{d_\rho}(2+\lfloor{\al}\rfloor)$ generates $G$. Lemma \ref{worddense} then implies that 
\begin{equation*}
    (1+\lfloor{\al}\rfloor-\al)d_{S_\rho}(x,y)-(n-1)\leq d_\rho(x,y)\leq (2+\lfloor{\al}\rfloor)d_{S_\rho}(x,y)
\end{equation*}
for all $x,y\in G$. Thus, we get that $\Del(\rho,[d_{S_\rho}])\leq R:=\log(\frac{2+\lfloor{\al}\rfloor}{1+\lfloor{\al}\rfloor-\al})$. In addition, we also get $h(G,S_\rho) \leq (2+\lfloor{\al}\rfloor)h(d_\rho)=2+\lfloor{\al}\rfloor$, and by applying Morse lemma, we can deduce that $(G,S_\rho)$ is $\wtilde{\del}$-hyperbolic, for $\wtilde{\del}$ depending only on $\del$ and $\al$. From this, we can apply Theorem \ref{thmBCGS} and conclude that there is a finite set $(G,S_1),\dots,(G,S_k)$ of marked groups such that each marked group $(G,S_\rho)$ with $\rho\in \scrD_{\del,\al}(G)$ is isometrically isomorphic to some $(G,S_i)$. Define $K$ as the intersection of $\scrD_{\del,\al}(G)$ with the closure of the union $B_\Del([d_{S_1}],R)\cup \cdots \cup B_\Del([d_{S_k}],R)$, which is compact by Proposition \ref{propersecond}. As marked isomorphisms are induced by automorphisms of $G$ exchanging the corresponding word metrics, and since $\Out(G)$ acts isometrically, we get $\scrD_{\del,\al}(G)=\Out(G)\cdot K$ and conclude the proof of the theorem.
\end{proof}


\section{Continuity of the Bowen-Margulis current and the mean distortion}\label{secBM}
In this section we prove Theorems \ref{BMctns} and \ref{meandistctns}. We start with the continuity of the Bowen-Margulis current, so we fix a sequence $(\rho_n)_n\subset \scrD_{\del}(G)$ converging to $\rho_\infty$. Our goal is to prove that $BM(\rho_n)$ converges to $BM(\rho_\infty)$ in $\bbP\Curr(G)$.
As $\bbP\Curr(G)$ is metrizable, our strategy will be to prove that each subsequence of $(BM(\rho_{n}))_n$ has a further subsequence converging to $BM(\rho_\infty)$.

Fix a finite symmetric generating set $S\subset G$. By Proposition \ref{boundimpliesroughgeod}, there exists some $\al\geq 0$ such that $\rho_n \in \scrD_{\del,\al}(G)$ for all $n$, and by Lemma \ref{lem.unifquasiisom} we can find a constant $C\geq 1$ and $\del$-hyperbolic, $\al$-rough geodesic pseudometrics $d_n\in \hat\rho_n$ such that 
\begin{equation}\label{eqqitod}
    C^{-1}d_S(x,y)-C\leq d_n(x,y)\leq Cd_S(x,y)
\end{equation}
for all $n$ and all $x,y\in G$.

We consider a subsequence of $(\rho_n)_n$, that for simplicity we still denote by $\rho_n$. Due to our assumptions, and after extracting a further subsequence, we can assume that $d_{n}$ converges pointwise to a pseudometric $d_\infty\in \hat\rho_\infty$, so that $d_\infty$ is also $\del$-hyperbolic and $\al$-rough geodesic and satisfies \eqref{eqqitod} with $d_\infty$ in the place of $d_n$. 

The next lemma will guarantee a good large-scale convergence of the pseudometrics $d_{n}$, and might be considered as a characterization of convergence of metric structures in terms of pointwise convergence of convenient pseudometric representatives.

\begin{lemma}\label{lem.pointconvimpliesunifconv}
Let $\del,\al\geq 0$ and $(d_n)_n \subset \calD(G)$ be a sequence of $\del$-hyperbolic, $\al$-rough geodesic pseudometrics on $G$. Suppose that $d_n$ converges pointwise to a pseudometric $d_\infty$, also $\del$-hyperbolic and $\al$-rough geodesic. Then for any $\ep\in(0,1)$ there is some $n_0$ such that if $n\geq n_0$ and $x\in G$ then
\begin{equation*}
    d_n(x,1) \leq (1+\ep)d_\infty(x,1)+3\al+1.
\end{equation*}
In addition, suppose that there exist $A,B\geq 1$ such that $d_\infty \leq Ad_n+B$ for all $n$. Then for any $\ep\in(0,1)$ there is some $n_1$ such that if $n\geq n_1$ and $x\in G$ then
\begin{equation*}
    d_\infty(x,1) \leq (1+\ep)d_n(x,1)+3\al+1.
\end{equation*}
\end{lemma}
\begin{proof}
Let $\ep\in(0,1)$, and for $x\in G$ consider an $\al$-rough geodesic $1=x_0,\dots,x_m=x$ for $d_\infty$ joining $1$ and $x$. Given $L>\frac{2\al+\ep}{2\ep}$ an integer, let $n_0$ be such that $|d_n(p,q)-d_\infty(p,q)|<\ep/2$ for all $n\geq n_0$ and all $p,q$ such that $d_\infty(p,q)\leq L+\al$. If $m=sL+r$
with $s,r$ nonnegative integers and $0 \leq r<L$, then
\begin{align*}
    d_n(1,x) & \leq \sum_{i=0}^{s-1}{d_n(x_{iL},x_{(i+1)L})}+d_n(x_{sL},x_m) \\
    & \leq \sum_{i=0}^{s-1}{\left(d_\infty(x_{iL},x_{(i+1)L})+\ep/2\right)}+d_\infty(x_{sL},x_m)+\ep/2 \\
     & \leq s(L+\al+\ep/2)+(r+\al+\ep/2)\\ &=m+(s+1)(\al+\ep/2)\\ &\leq d_\infty(1,x)+\al+(s+1)(\al+\ep/2).
\end{align*}
Since $s\leq d_{\infty}(1,x)/L+\al/L$ we deduce
$$d_n(1,x)\leq (1+\al/L+\ep/(2L))d_\infty(1,x)+\al+(\al/L+1)(\al+\ep/2)\leq (1+\ep)d_\infty(1,x)+3\al+1,$$
where we used $\al/L+\ep/(2L)<\ep<1$. This proves the first statement. 

The second statement is proven similarly, where we choose $n_1$ such that $|d_n(p,q)-d_\infty(p,q)|<\ep/2$ for all $n\geq n_1$ and all $p,q$ such that $d_\infty(p,q)\leq A(L+\al)+B$, where $L>\frac{2\al+\ep}{2\ep}$ is an integer, and $x_0,\dots, x_m$ is now an $\al$-rough geodesic for $d_n$ with $n\geq n_1$.
\end{proof}

After extracting a subsequence, by Lemma \ref{lem.pointconvimpliesunifconv} we can assume that there is a decreasing sequence $\Lam_n \to 1$ such that 
\begin{equation}\label{distconv}
    \Lam_n^{-1}d_\infty-3\al-2 \leq d_{n} \leq \Lam_nd_\infty+3\al+1
\end{equation}
for all $n$.
By Proposition \ref{strongqi} this implies that 
\begin{equation}\label{groprodconv}
    \Lam_n^{-1}\groprod{.}{.}{.,d_\infty}-Q \leq \groprod{.}{.}{.,d_n} \leq \Lam_n\groprod{.}{.}{.,d_\infty}+Q
\end{equation}
for all $n$, where $Q$ is a constant independent of $n$.

Given $n$, let $\nu_n$ be a quasiconformal measure for $d_n$ with uniform multiplicative constant $D=D_\del$ as in \eqref{equnifqc}. As $\partial G$ is compact metrizable, and after extracting a subsequence, we can assume that $\nu_n$ weak-$*$ converges to $\nu_\infty$. We claim that $\nu_\infty$ is quasiconformal for $d_\infty$.

\begin{prop}\label{convtoqc}
There is a constant $M\geq 1$ such that for all $x\in G$ and all $f\in C(\partial G)$ with $f\geq 0$:
\begin{equation}\label{eqqcintegral}
   M^{-1}\int{f(p)e^{-\beta_{d_\infty}(x,1;p)}d\nu_\infty(p)} \leq  \int{f(p)dx\nu_\infty(p)}\leq M\int{f(p)e^{-\beta_{d_\infty}(x,1;p)}d\nu_\infty(p)}.
\end{equation}
In consequence, $\nu_\infty$ is quasiconformal for $d_\infty$.
\end{prop}

We will need the next lemma, which is an immediate consequence of \eqref{gromovinboundary}.
\begin{lemma}\label{approxgroprodbyctns}
Let $d\in \calD(G)$ be $\del$-hyperbolic. Then for all $x\in G$, the functions $\ov{h}_x,\underline{h}_x:\partial G\ra \R$ defined by 
$$\ov{h}_x(p):=\limsup_{q\to p}{e^{\groprod{x}{q}{1,d}}},\hspace{5mm}\underline{h}_x(p):=\liminf_{q\to p}{e^{\groprod{x}{q}{1,d}}}$$
are upper and lower semi-continuous respectively, bounded above by $e^{d(x,1)}$, and satisfy the inequalities
\begin{equation*}
   \underline{h}_x(p)\leq \ov{h}_x(p), \hspace{5mm}    e^{-3\del}\ov{h}_x(p)\leq e^{\groprod{x}{p}{1,d}}\leq e^{3\del}\underline{h}_x(p)  
\end{equation*}
for all $p\in \partial G$.
\end{lemma}

\begin{proof}[Proof of Proposition \ref{convtoqc}]
Quasiconformality follows immediately from \eqref{eqqcintegral} since $\nu_\infty$ is regular. We will only prove the second inequality in \eqref{eqqcintegral}, as the first one is proven in the same way. Let $f\in C(\partial G)$ with $f\geq 0$ and $x\in G$. By inequalities \eqref{equnifqc} and \eqref{eqbuseq} we get
\begin{align*}
    \int{fdx\nu_\infty}& =\lim_{n\to \infty}{ \int{fdx\nu_n}}=\lim_{n\to \infty}{ \int{f\frac{dx\nu_n}{d\nu_n}d\nu_n}} \\
    & \leq D\limsup_{n\to \infty}{ \int{f(p)e^{-\beta_{d_n}(x,1;p)}d\nu_n(p)}} \\
    & \leq De^{2\del}\limsup_{n\to \infty}{ \left(e^{-d_n(x,1)}\int{f(p)e^{2\groprod{x}{p}{1,d_n}}d\nu_n(p)}\right)}.
\end{align*}
Since the sequence $(\Lam_n)_n$ is monotone, by \eqref{distconv} and \eqref{groprodconv} we obtain
\begin{align*}
    \int{fdx\nu_\infty} & \leq D e^{2\del+3\al+2Q+2}\limsup_{n\to \infty}{ \left(e^{-\Lam_n^{-1}
    d_\infty(x,1)}\int{f(p)e^{2\Lam_n\groprod{x}{p}{1,d_\infty}}d\nu_n(p)}\right)}\\
    & \leq D e^{2\del+3\al+2Q+2}e^{-
    d_\infty(x,1)}\liminf_{k\to \infty}{\limsup_{n\to \infty}{ \left(\int{f(p)(e^{2\groprod{x}{p}{1,d_\infty}})^{\Lam_k}d\nu_n(p)}\right)}}.
\end{align*}

At this point, we would like to use the convergence $\nu_n \stackrel{\ast}{\rightharpoonup} \nu_\infty$ to deduce $$\limsup_{n\to \infty}{ \left(\int{f(p)(e^{2\groprod{x}{p}{1,d_\infty}})^{\Lam_k}d\nu_n(p)}\right)}= \int{f(p)(e^{2\groprod{x}{p}{1,d_\infty}})^{\Lam_k}d\nu_\infty(p)}.$$
However, the function $p \to (e^{2\groprod{x}{p}{1,d_\infty}})^{\Lam_k}$ is not necessarily continuous, so this last equation might not hold. Instead, we consider the function $\ov{h}=\ov{h}_x:\partial G\ra \R$ from 
Lemma \ref{approxgroprodbyctns} with respect to $d=d_\infty$. Since $\partial G$ is metrizable and $\ov{h}$ is upper semi-continuous, there exists a decreasing sequence $\ov{h}_m$ of continuous functions on $\partial G$ that converges pointwise to $\ov{h}$. Moreover, since $\ov{h}$ is bounded, we can assume that the sequence $\ov{h}_m$ is uniformly bounded. Therefore, the convergence $\nu_n \stackrel{\ast}{\rightharpoonup} \nu_\infty$ and the dominated convergence theorem yields
\begin{align*}
    \limsup_{n\to \infty}{ \left(\int{f(p)(e^{2\groprod{x}{p}{1,d_\infty}})^{\Lam_k}d\nu_n(p)}\right)} & \leq  e^{6\Lam_k\del}\limsup_{n\to \infty}{ \left(\int{f\cdot(\ov{h})^{2\Lam_k}d\nu_n}\right)}\\
     &\leq  e^{6\Lam_k\del}\liminf_{m\to \infty}{\limsup_{n\to \infty}{ \left(\int{f\cdot(\ov{h}_m)^{2\Lam_k}d\nu_n}\right)}}\\
    &=  e^{6\Lam_k\del}\liminf_{m\to \infty}{ \left(\int{f\cdot (\ov{h}_m)^{2\Lam_k}d\nu_\infty}\right)}\\
    &=  e^{6\Lam_k\del} \int{f\cdot(\ov{h})^{2\Lam_k}d\nu_\infty}\\
     & \leq   e^{12\Lam_k\del} \int{f(p)(e^{2\groprod{x}{p}{1,d_\infty}})^{\Lam_k}d\nu_\infty(p)}.
\end{align*}
Combining these two inequalities, and applying the dominated convergence theorem to the sequence $k\mapsto (p\mapsto f(p)(e^{2\groprod{x}{p}{1,d_\infty}})^{\Lam_k})$ we deduce
\begin{align*}
    \int{fdx\nu_\infty} & \leq D e^{14\del+3\al+2Q+2}\int{f(p)e^{-
    d_\infty(x,1)+2\groprod{x}{p}{1,d_\infty}}d\nu_\infty(p)} \\ 
    & \leq D e^{16\del+3\al+2Q+2}\int{f(p)e^{-\beta_{d_
    \infty}(x,1;p)}d\nu_\infty(p)},
\end{align*}
where in the last inequality we used \eqref{eqbuseq}. Since $M=D e^{16\del+3\al+2Q+2}$ is independent of $f$ and $x$, the conclusion follows.
\end{proof}

As the quasiconformal measures $\nu_n$ satisfy \eqref{equnifqc} and each $d_n$ is $\del$-hyperbolic, there is a constant $R$ depending only on $\del$ and $D_\del$ such that for every $n$ there is a measure $m_n$ representing $BM(\rho_n)$ and satisfying
\begin{equation}\label{eqapproxbm}
    R^{-1}\int_A{e^{2\groprod{p}{q}{1,d_n}}d\nu_n(p)d\nu_n(q)} \leq  m_n(A)\leq R\int_A{e^{2\groprod{p}{q}{1,d_n}}d\nu_n(p)d\nu_n(q)}
\end{equation}
for any Borel subset $A\subset \partial^2 G$, see e.g. \cite[Sec.~3]{Furman2002coarsenegcurv}.
\begin{lemma}
After taking a subsequence, $m_n$ weak-$*$ converges to a positive Radon measure $\om$.
\end{lemma}
\begin{proof}
Since $\partial^2 G$ is locally compact and metrizable, it is $\sigma$-compact, so it is enough to show that for any compact set $C\subset \partial^2 G$ with non-empty interior, the sequence $(m_n(C))_n$ is bounded by above and below by positive constants.
For the upper bound, we apply \eqref{eqapproxbm},  \eqref{groprodconv} and the monotonicity of $(\Lam_n)_n$ to obtain
\begin{align*}
    m_n(C)& \leq R\int_C{e^{2\groprod{p}{q}{1,d_n}}d\nu_n(p)d\nu_n(q)}\\
    & \leq R\int_C{e^{2\Lam_1\groprod{p}{q}{1,d_\infty}+2Q}d\nu_n(p)d\nu_n(q)}\\ & \leq R\sup_{(p,q)\in C}{\left(e^{2\Lam_1\groprod{p}{q}{1,d_\infty}+2Q}\right)}(\nu_n\otimes \nu_n)(C) \leq R\sup_{(p,q)\in C}{\left(e^{2\Lam_1\groprod{p}{q}{1,d_\infty}+2Q}\right)}.
\end{align*}
The last quantity is finite by compactness of $C$ and is independent of $n$. 

For the lower bound, we may assume that $C$ contains a non-empty open set $U=U_1\times U_2$ with each $U_i\subset \partial G$, so that the convergence $\nu_n \stackrel{\ast}{\rightharpoonup} \nu_\infty$ gives
\begin{align*}
    \liminf_{n}{m_n(C)} \geq  \liminf_{n}{m_n(U)} &\geq R^{-1}\liminf_n{\nu_n(U_1)\nu_n(U_1)}\geq R^{-1}\nu_\infty(U_1)\nu_{\infty}(U_2). 
\end{align*}
The last term is positive since the measure $\nu_\infty$ has full support, and hence the sequence $(m_n(C))_n$ is bounded by below by a positive number.
\end{proof}

Now we finish the proof of Theorem \ref{BMctns}, for which we are left to show that $\om$ represents $BM(\rho_\infty)$.
Indeed, given $f\in C_c(\partial^2 G)$ with $f\geq 0$, by \eqref{eqapproxbm} and \eqref{groprodconv} we have 
\begin{align*}
    \int{fd\om} & =\lim_{n\to \infty}{\int{fdm_n}}\leq R\limsup_{n\to \infty}{\int{f(p,q)e^{2\groprod{p}{q}{1,d_n}}d\nu_n(p)d\nu_n(q)}}  \\ & \leq Re^{2Q}\liminf_{k\to \infty}{\limsup_{n\to \infty}{\left(\int{f(p,q)e^{2\Lam_k\groprod{p}{q}{1,d_\infty}}d\nu_n(p)d\nu_n(q)}\right)}}. 
\end{align*}

Let $\ep>0$ be such $\ep\del<\frac{\log{2}}{6}$, and let $\varrho_\ep$ be a visual metric on $\partial G$ satisfying \eqref{eqvismet} with respect to $d_\infty$. As $f$ has compact support in $\partial^2 G$, the functions $(p,q)\mapsto f(p,q)(\varrho_{\ep}(p,q)^{-1/\ep})^{2\Lam_k}$ are continuous on $(\partial G)^2$ and uniformly bounded. Also, since $\nu_n \stackrel{\ast}{\rightharpoonup} \nu_\infty$ and $\partial G$ is separable, $\nu_n\otimes \nu_n$ weak-$*$ converges to $\nu_\infty \otimes \nu_\infty$  \cite[Thm.~2.8]{Billingsley1999Convprob}. Therefore, by the dominated convergence theorem we obtain
\begin{align*}
    \int{fd\om} & \leq Re^{2Q}\liminf_{k\to \infty}{\limsup_{n\to \infty}{\left(\int{f(p,q)(\varrho_{\ep}(p,q)^{-1/\ep})^{2\Lam_k}d\nu_n(p)d\nu_n(q)}\right)}} \\
    & \leq Re^{2Q}\liminf_{k\to \infty}{ \left(\int{f(p,q)(\varrho_{\ep}(p,q)^{-1/\ep})^{2\Lam_k}d\nu_\infty(p)d\nu_\infty(q)}\right)} \\
    & \leq Re^{2Q}\int{f(p,q)\varrho_{\ep}(p,q)^{-2/\ep}d\nu_\infty(p)d\nu_\infty(q)} \\
    & \leq Re^{2Q}(3-2e^{3\ep\del})^{-2/\ep}\int{f(p,q)e^{2\groprod{p}{q}{1,d_\infty}}d\nu_\infty(p)d\nu_\infty(q)}.
\end{align*}
As $\nu_\infty$ is quasiconformal for $d_\infty$, we can find a measure $m_\infty$ representing $BM(\rho_\infty)$ in the same class of $e^{2\groprod{p}{q}{1,d_\infty}}d\nu_\infty(p)d\nu_\infty(q)$ and with essentially uniformly bounded Radon-Nikodym derivatives, and hence there is a constant $L\geq 1$ independent of $f$ such that $$\int{fd\om}\leq L\int{fdm_\infty}.$$

In the same way, we can prove that there is a constant $L'$ such that $\int{fdm_\infty}\leq L'\int{fd\om}$, implying that $\om$ and $m_\infty$ are absolutely continuous with respect to each other, and that the Radon-Nikodym derivative $\frac{d\om}{dm_\infty}$ is essentially bounded by above, and by below by a positive number. To conclude the proof of Theorem \ref{BMctns}, note that $\om$ is $G$-invariant, being the limit of $G$-invariant measures, and that $m_\infty$ is $G$-ergodic \cite[Thm.~1.4]{BaderFurman2017Ergodichyp}. In consequence, $\frac{d\om}{dm_\infty}$ is constant almost everywhere, and there is some $\lam>0$ such that $\om=\lam m_\infty$, so that $BM(\rho_n)\to BM(\rho_\infty)$. \qed

We end with this section with the proof of Theorem \ref{meandistctns}, which follows from the properties of the Manhattan curve, extensively studied in \cite{CantrellTanaka2021Manhattan}. Given $d,d'\in \calD(G)$, the \emph{Manhattan curve} for the pair $(d,d')$ is the function $\Thet:\R \ra \R$ which maps $a$ to the critical exponent of 
$$b \mapsto  \calP_{d,d'}(a,b)=\sum_{x\in G}{e^{-ad'(x,1)-bd(x,1)}}.$$
The Manhattan curve is decreasing, convex, and continuously differentiable, and its relation to the mean distortion is given by the identity 
\begin{equation}\label{tauTheta'}
    \tau(d'/d)=-\Thet'(0),
\end{equation}
see Theorems 1.1 and 3.12 in \cite{CantrellTanaka2021Manhattan}.
\begin{proof}[Proof of Theorem \ref{meandistctns}]
Consider $\rho,\rho'\in \scrD(G)$ and sequences $(\rho_n)_n,(\rho'_n)_n$ in $\scrD(G)$ converging to $\rho$ and $\rho'$ respectively, for which we claim that $\tau(\rho'_n/\rho_n)$ tends to $\tau(\rho'/\rho)$. For each $n\geq 1$ choose representatives $d_{n}\in \hat{\rho}_n$ and $d_n'\in \hat{\rho}'_n$, and let $\Thet_n$ be the Manhattan curve for the pair $(d_n,d'_n)$. Similarly, choose $d\in \hat{\rho}$ and $d'\in \hat{\rho}'$ and let $\Thet$ be the Manhattan curve for $(d,d')$. From \eqref{tauTheta'} we obtain $\tau(\rho'_n/\rho_n)=-\Thet'_n(0)$ and $\tau(\rho'/\rho)=-\Thet'(0)$, so it is enough to show that $\Thet'_n$ converges to $\Thet'$ pointwise.

Let $(\Lam_n)_n$ and $(C_n)_n$ be sequences such that $\Lam_n\to 1$ and satisfying 
\begin{equation}\label{compmetrics}
    \Lam_n^{-1}d-C_n\leq d_n\leq \Lam_nd+C_n \hspace{5mm} \text{and} \hspace{5mm} \Lam_n^{-1}d'-C_n\leq d'_n\leq \Lam_nd'+C_n
\end{equation}
for all $n$. For any $a,b\in \R$, from \eqref{compmetrics} we get 
$$e^{-(|a|+|b|)C_n}\calP_{d,d'}(\max(\Lam_na,\Lam_n^{-1}a),\max(\Lam_nb,\Lam_n^{-1}b))\leq \calP_{d_n,d'_n}(a,b),$$
so that \begin{equation*}
    \Thet(\max(\Lam_na,\Lam_n^{-1}a))\leq \max(\Lam_n\Thet_n(a),\Lam_n^{-1}\Thet_n(a)),
\end{equation*}
and by the same argument we also get 
\begin{equation*}
    \Thet(\min(\Lam_na,\Lam_n^{-1}a))\geq \min(\Lam_n\Thet_n(a),\Lam_n^{-1}\Thet_n(a)).
\end{equation*}
Continuity of $\Thet$ then implies that $\Thet_n$ converges pointwise to $\Thet$, and since $\Thet$ is differentiable and all the curves $\Thet_n$ are convex and differentiable, we conclude that $\Thet_n'$ converges pointwise to $\Thet'$, as desired.
\end{proof}

\medskip
\hspace{-4.3mm}\textbf{Acknowledgment}\hspace{2mm} 
I am grateful to Ian Agol for very interesting discussions throughout all this work. I also thank Stephen Cantrell and Ryokichi Tanaka for the valuable correspondence and useful comments about the Manhattan curve and mean distortion, Peter Haïssinsky for the corrections, and Chi Cheuk Tsang for the careful reading. My thanks to the anonymous referee for very helpful suggestions.

\small{Eduardo Oreg\'on-Reyes (\texttt{eoregon@berkeley.edu})}\\
\small{Department of Mathematics}\\
\small{University of California at Berkeley}\\
\small{1087 Evans Hall, Berkeley, CA 94720-3840, U.S.A.}


\begin{thebibliography}{26}





\bibitem{BaderFurman2017Ergodichyp}
U.~Bader, A.~Furman. Some ergodic properties of metrics on hyperbolic groups.
\url{https://arxiv.org/abs/1707.02020}, arXiv preprint, 2017.

\bibitem{BessonCourtoisGallotSambusetti2021Finitenesshyp}
G.~Besson, G.~Courtois, S.~Gallot, A.~Sambusetti. Finiteness Theorems for Gromov-Hyperbolic Spaces and Groups.
\url{https://arxiv.org/abs/2109.13025}, arXiv preprint, 2021.

\bibitem{Bestvina2002ICM}
M.~Bestvina, The topology of $\Out(F_n)$. \textit{Proceedings of the International Congress of Mathematicians-Beijin} 2002 \textbf{2}, 373–-384, Higher Ed. Press, Beijing, (2002).

\bibitem{Billingsley1999Convprob}
P.~Billingsley, Convergence of Probability Measures. Wiley, Second edition, 1999.

\bibitem{BlachereHaissinskyMathieu2011HarmvsQconf}
S.~Blach\`ere, P.~Haïssinsky, P.~Mathieu, Harmonic measures versus quasiconformal measures for hyperbolic groups, \textit{Ann. Sci. \'Ec. Norm. Sup\'er.} (4) \textbf{44} (2011), no. 4, 683–-721.(hal-00290127v2)

 
\bibitem{Bonahon1991currentsnegcurv}
F.~Bonahon, Geodesic currents on negatively curved groups. \textit{Arboreal group theory} (Berkeley, CA, 1988), 143–168, Math. Sci. Res. Inst. Publ., \textbf{19}, Springer, New York, 1991.

\bibitem{Bonahon1988GeomTeichcurrents}
F.~Bonahon, The geometry of Teichmüller space via geodesic currents. \textit{Invent. Math.} \textbf{92} (1988), 139--162.


\bibitem{BreuillardFujiwara2021JSRnonpos}
E.~Breuillard, K.~Fujiwara, On the joint spectral radius for isometries of non-positively curved spaces and uniform growth. \textit{Ann. Inst. Fourier} \textbf{71} (2021), no. 1, 317--391.

\bibitem{BridsonHaefliger1999}
M.~Bridson, A.~Haefliger, Metric spaces of non-positive curvature. Grundlehren der Mathematischen Wissenschaften, \textbf{319}, Springer-Verlag, Berlin, 1999.

\bibitem{CalegariFujiwara2010Combable}
D.~Calegari, K.~Fujiwara, Combable functions, quasimorphisms, and the central limit theorem. \textit{Ergodic Theory Dynam. Systems} \textbf{30} (2010), no. 5, 1343--1369.

\bibitem{CantrellTanaka2021Manhattan}
S.~Cantrell, R.~Tanaka. The Manhattan curve, ergodic theory of topological flows and rigidity.
\url{https://arxiv.org/abs/2104.13451}, arXiv preprint, 2021.

\bibitem{Clay2005ContractdefspaceGtrees}
M.~Clay, Contractibility of deformation spaces of $G$–trees. \textit{Algebr. Geom. Topol.} \textbf{5} (2005), 1481–-1503.

\bibitem{Coornaert1993PSmeasures}
M.~Coornaert, Mesures de Patterson-Sullivan sur le bord d'un espace hyperbolique au sens de Gromov. \textit{Pacific J. Math.} \textbf{159} (1993), 241--270.

\bibitem{CullerVogtmann1986Outerspace}
M.~Culler, K.~Vogtmann, Moduli of graphs and automorphisms of free
groups. \textit{Invent. Math.} \textbf{84} (1986), no. 1, 91-–119.

\bibitem{DahmaniHorbez2018RandomwalksMCG+Out}
F.~Dahmani, C.~Horbez, Spectral theorems for random walks on mapping class groups and $\Out(F_N)$. \textit{Int. Math. Res. Not.} \textbf{9} (2018), 2693-–2744.

\bibitem{DavisJanuszkiewicz1991hyppolyhedra}
M.~W.~Davis, T.~Januszkiewicz. Hyperbolization of polyhedra. \textit{J. Differential Geom.}, \textbf{34} (1991), no. 2, 347–-388. 

\bibitem{DelzantGuichardLabourieMozes2011Displacing}
T.~Delzant, O.~Guichard, F.~Labourie, S.~Mozes. Displacing representations and orbit maps. In: Farb, B., Fisher, D., Zimmer, R.J. (eds.) Geometry, Rigidity, and Group Actions, pp. 494–514. University of Chicago Press, Chicago, 2011.

\bibitem{DeyKapovich2019PSAnosov}
S.~Dey, M.~Kapovich. Patterson–Sullivan theory for Anosov subgroups. \url{https://arxiv.org/abs/1904.10196}, arXiv preprint, 2019.

\bibitem{ErlandssonParlierSouto2020Stablelengthhyp}
V.~Erlandsson, H.~Parlier, J.~Souto, Counting curves, and the stable length of currents. \textit{J. Eur. Math. Soc.} \textbf{22} (2020), no. 6, 1675-–1702.

\bibitem{FarbMargalit2012primer}
B.~Farb, D.~Margalit, A Primer on Mapping Class Groups. Princeton Mathematical Series, \textbf{49}, Princeton University Press, 2012.

\bibitem{FrancavigliaMartino2011Metricouter}
S.~Francaviglia, A.~Martino, Metric properties of outer space. \textit{Publ. Mat.} \textbf{55} (2011), no. 2, 433–-473.

\bibitem{FrickerFurman2021QFuchsvsNegcurv}
E.~Fricker, A.~Furman. Quasi-Fuchsian vs Negative curvature metrics on surface groups. \url{https://arxiv.org/abs/2108.07947}, arXiv preprint, 2021.

\bibitem{Furman2002coarsenegcurv}
A.~Furman, Coarse-geometric perspective on negatively curved manifolds and groups. Rigidity in dynamics and geometry (Cambridge, 2000), pp. 149–166. Springer, Berlin, 2002.


\bibitem{GhydelaHarpe1990book}
\'E.~Ghys, P.~de~la~Harpe, Sur les groupes hyperboliques d'apr\`es Mikhael Gromov. Progress in Mathematics, \textbf{83}. Birkhäuser Boston, Inc., Boston, MA, 1990.

\bibitem{GouezelMatheusMaucourant2018Entropydrifthyp}
S.~Gouëzel, F.~Math\'eus, F.~Maucourant, Entropy and drift in word hyperbolic groups. \textit{Invent. Math.} \textbf{211} (2018), 1201-–1255.

\bibitem{GrovesManning2008DehnFilling}
D.~Groves, J.~F.~Manning, Dehn filling in relatively hyperbolic groups. \textit{Israel J. Math.} \textbf{168} (2008), 317-–429.

\bibitem{GuirardelHorbez2021MeasEquivRigOut(Fn)}
V.~Guirardel, C.~Horbez. Measure equivalence rigidity of $\Out(F_N)$.
\url{https://arxiv.org/abs/2103.03696}, arXiv preprint, 2021.

\bibitem{GuirardelLevitt2007Outerspacefreeproduct}
V.~Guirardel, G.~Levitt, The outer space of a free product.
\textit{Proc. Lond. Math. Soc.} (3), \textbf{94} (2007), no. 3, 695–-714.



\bibitem{Kaimanoch1990Invmeasgeodesicflow}
V.~A.~Kaimanovich, Invariant measures of the geodesic flow and measures at infinity on negatively curved manifolds. \textit{Ann. Inst. Henri Poincar\'e} \textbf{53} (1990), no. 4, 361–-393.

\bibitem{KaimanovichKapovichSchupp2007Genstretchingfactor}
V.~A.~Kaimanovich, I.~Kapovich, P.~Schupp, The subadditive ergodic theorem and generic stretching factors for free group automorphisms. \textit{Israel J. Math.} \textbf{157} (2007), 1--46.

\bibitem{KapovicNagnibeda2007PSembedding+minvol}
I.~Kapovich, T.~Nagnibeda, The Patterson–Sullivan embedding and minimal volume entropy for outer space. \textit{Geom. Funct. Anal.} \textbf{17} (2007), no. 4, 1201–-1236.

\bibitem{LatokKnieperPollicottWeiss1989DiffentropyAnosov}
A.~Katok, G.~Knieper, M.~Pollicott, H.~Weiss, Differentiability and analyticity of topological entropy for Anosov and geodesic flows. \textit{Invent. Math.} \textbf{98} (1989), 581-–597.

\bibitem{KosenkoTiozzi2021FundineqFuchsian}
P.~Kosenko, G.~Tiozzo. The fundamental inequality for cocompact Fuchsian groups. \url{https://arxiv.org/abs/2012.07417}, arXiv preprint, 2021.

\bibitem{Krat2001Pairsmetrics}
S.~A.~Krat, On pairs of metrics invariant under a cocompact action of a group. \textit{Electron. Res. Announc. Amer. Math. Soc.} \textbf{7} (2001), 79-–86.

\bibitem{Lang2013injectivehull}
U.~Lang, Injective hulls of certain discrete metric spaces and groups. \textit{J. Topol. Anal.}, \textbf{5} (2013), no. 3, 297–-331.

\bibitem{NicaSpakula2016stronghyp}
B.~Nica, J.~\v{S}pakula, Strong hyperbolicity. \textit{Groups Geom. Dyn.} \textbf{10} (2016), no. 3, 951–-964.

\bibitem{OregonReyes18propsetsisom}
E. Oreg\'on-Reyes, Properties of sets of isometries of Gromov hyperbolic spaces. \textit{Groups Geom. Dyn.} \textbf{12} (2018), no. 3, 889--910.

\bibitem{Osin2007Perfilling}
D.~V.~Osin, Peripheral fillings of relatively hyperbolic groups. \textit{Invent. Math.} \textbf{167} (2007), no. 2, 295-—326.

\bibitem{Otal1990MLSRigidity}
J.-P.~Otal, Le spectre marqu\'e des longueurs des surfaces \`a courbure n\'egative.
(French)
\textit{Ann. of Math.} (2) \textbf{131} (1990), no. 1, 151--162.

\bibitem{Otal1998HypHaken}
J.-P.~Otal, Thurston's hyperbolization of Haken manifolds. In \textit{Surveys in Differential Geometry}, Vol. III, 77-194, Internat. Press, Boston, MA, 1998.

\bibitem{Paulin1991OuthypRtrees}
F.~Paulin, Outer automorphisms of hyperbolic groups and small actions on \textbf{R}-trees. \textit{Arboreal group theory} (Berkeley, CA, 1988), 331--343, Math. Sci. Res. Inst. Publ., \textbf{19}, Springer, New York, 1991.

\bibitem{Paulin1988TopGromovRtrees}
F.~Paulin, Topologie de Gromov \'equivariante, structures hyperboliques et arbres r\'eels. \textit{Invent. Math.} \textbf{94} (1988), 53–-80.



\bibitem{Thurston1998Minstretchmaps}
W.~Thurston. Minimal stretch maps between hyperbolic surfaces.
\url{https://arxiv.org/abs/math/9801039}, arXiv preprint, 1998.

\bibitem{Wienhard2019ICM}
A.~Wienhard, An invitation to higher Teichmüller theory. \textit{Proceedings of the International Congress of Mathematicians-Rio de Janeiro} 2018 \textbf{2}, 1007–-1034, World Scientific, (2019).


\end{thebibliography}
\end{document}